\documentclass[11pt,a4paper]{article}
\usepackage[utf8]{inputenc}
\usepackage[a4paper, total={6in, 8in}]{geometry}
\usepackage[english]{babel}
\usepackage{theorem}
\usepackage{amsmath}
\numberwithin{equation}{section}
\let\theoremstyle\undefined
\usepackage{amsthm}
\usepackage{hyperref}
\usepackage{lipsum}
\usepackage{amsfonts}
\usepackage{amssymb}
\usepackage{graphicx}
\graphicspath{{images/}}
\usepackage{epstopdf}
\usepackage[normalem]{ulem} 
\usepackage{algorithmic}
\usepackage{comment}
\ifpdf
  \DeclareGraphicsExtensions{.eps,.pdf,.png,.jpg}
\else
  \DeclareGraphicsExtensions{.eps}
\fi

\usepackage{multirow}

\usepackage[noabbrev]{cleveref}
\crefname{equation}{equation}{equations}

\usepackage{enumitem}
\setlist[enumerate]{leftmargin=.5in}
\setlist[itemize]{leftmargin=.5in}

\newtheorem{theorem}{Theorem}[]
\newtheorem{definition}[theorem]{Definition}

\newtheorem{lemma}[theorem]{Lemma}
\newtheorem{proposition}[theorem]{Proposition}
\newtheorem{corollary}[theorem]{Corollary}
\newtheorem{remark}{Remark}[section]

\newcounter{MotEx}

{
  \theoremstyle{plain}
  \theoremheaderfont{\normalfont}
  \theorembodyfont{\normalfont\itshape}
  \newtheorem{MotExample}[MotEx]{\textbf{Motivating Example}, Step}
}

\usepackage{tikz}
\usepackage{tikz-3dplot}
    \usetikzlibrary{shapes.geometric, arrows, positioning, calc, fadings, 3d}

\def\<{\langle}
\def\>{\rangle}

\DeclareFontFamily{U}{mathx}{}
\DeclareFontShape{U}{mathx}{m}{n}{<-> mathx10}{}
\DeclareSymbolFont{mathx}{U}{mathx}{m}{n}
\DeclareMathAccent{\widehat}{0}{mathx}{"70}
\DeclareMathAccent{\widecheck}{0}{mathx}{"71}

\DeclareMathOperator{\R}{\mathbb{R}}

\DeclareMathOperator{\N}{\mathbb{N}}
\DeclareMathOperator{\Z}{\mathbb{Z}}

\newcommand{\norm}[1]{\left\lVert#1\right\rVert}
\newcommand{\abs}[1]{\left|#1\right|}

\DeclareMathOperator{\supp}{supp}
\DeclareMathOperator*{\argmin}{argmin}

\DeclareMathOperator{\Id}{Id}

\DeclareMathOperator{\RadonD}{\vec{\mathcal{R}}}
\DeclareMathOperator{\radonSD}{\mathcal{R}^{\text{sd}}}
\DeclareMathOperator{\cradonSD}{\mathscr{R}^{\text{sd}}}

\newcommand{\Npxl}{N_{\text{pxl}}}
\newcommand{\Ndtc}{N_{\text{dtc}}}

\newcommand{\m}{\vec{m}} 
\newcommand{\jkm}{{j, k, \m}} 
\newcommand{\gen}{\psi} 
\newcommand{\hatgen}{\widehat{\gen}}

\newcommand{\bump}{\widehat{\phi}} 
\newcommand{\SHsys}{\Psi} 

\newcommand{\dimFlip}{M_{\! \leftrightarrow}}

\newcommand{\bapu}{\varphi} 
\newcommand{\videos}{\mathcal{G}^2(Z)} 

\newcommand{\mcP}{\mathcal{P}}  
\newcommand{\mcQ}{\mathcal{Q}}  
\newcommand{\mcS}{\mathcal{S}}  
\newcommand{\mcT}{\mathcal{T}}  
\newcommand{\mcY}{\mathcal{Y}}

\usepackage[mathscr]{eucal} 
\newcommand{\msR}{\mathscr{R}}

\newcommand{\msT}{\mathscr{T}} %
\newcommand{\msK}{\mathscr{K}}

\DeclareMathOperator{\sh}{\mathcal{SH}}
\DeclareMathOperator{\shD}{\vec{\sh}}

\renewcommand{\vec}[1]{\boldsymbol{#1}} 

\newcommand{\f}{\vec{f}} 
\newcommand{\g}{\vec{g}} 
\newcommand{\Wop}{\vec{W}} 
\newcommand{\Kop}{\vec{K}} 
\newcommand{\thetab}{\vec{\theta}} 
\newcommand{\gNd}{{\g_{N}^{\delta}}} 
\newcommand{\faNd}{\f_{\alpha,N}^{\delta}} 

\newcommand{\Prob}{\mathbb{P}}

\newcommand{\bu}{{\bf u}}
\newcommand{\be}{{\vec{\epsilon}}}

\newcommand{\cAbu}{\mathscr{A}_{\bu}}
\newcommand{\cA}{\mathscr{A}}
\newcommand{\cAm}{\mathscr{A}_\mu}
\newcommand{\Sbu}{S_{{\bu}}}
\newcommand{\E}{\mathbb{E}}
\def\fda{f^\delta_{\alpha}}
\def\fdan{f^\delta_{\alpha,N}}

\def\sdiff{r}

\def\Jda{J^\delta_{\alpha}}
\def\Jdan{J^\delta_{\alpha,N}}

\def\gd{{g}^{\delta}}
\def\gdan{{\bf g}^{\delta}_{N}}

\newcommand{\ip}[2]{\langle#1,#2\rangle}

\allowdisplaybreaks

\usepackage{bbm}
\DeclareRobustCommand{\bOne}{
  \text{\usefont{U}{bbold}{m}{n}1}} 

\newcommand{\bigger}[1]{\scalebox{1.2}{#1}} 

\newcounter{cnt}
\newcommand{\addPDFfig}[1]{
	\setcounter{cnt}{#1}
	\stepcounter{cnt}
	
	\includegraphics[width=.99\linewidth, page=\value{cnt}]{TikZfigures.pdf}
}

\newcommand{\email}[1]{\protect\href{mailto:#1}{#1}}

\def\stripzero#1{\expandafter\stripzerohelp#1}
\def\stripzerohelp#1{\ifx 0#1\expandafter\stripzerohelp\else#1\fi}

\title{Regularization with  optimal  space-time priors }

\date{Published in SIAM Journal on Imaging Sciences, 18(3), p. 1563-1600, 2025.\\ \href{https://doi.org/10.1137/24M1661923}{DOI: 10.1137/24M1661923}}

\author{Tatiana A. Bubba\thanks{Department of Mathematics and Computer Science, University of Ferrara, Italy (\email{tatiana.bubba@unife.it}).}
\and Tommi Heikkil\"a\thanks{School of Engineering Sciences, LUT University, Finland  (\email{tommi.heikkila@lut.fi}).}
\and Demetrio Labate\thanks{Department of Mathematics, University of Houston, USA (\email{dlabate@math.uh.edu}).}
\and Luca Ratti\thanks{Department of Mathematics, University of Bologna, Italy (\email{luca.ratti5@unibo.it}).}
}

\newlength{\plotSz}
\newlength{\imSz}
\newlength{\imShift}

\begin{document}

\maketitle

\begin{abstract}
We propose a variational regularization approach based on a multiscale representation called cylindrical shearlets aimed at dynamic imaging problems, especially dynamic tomography. The intuitive idea of our approach is to integrate a sequence of separable static problems in the mismatch term of the cost function, while the regularization term handles the nonstationary target as a spatio-temporal object. 
This approach is motivated by the fact that cylindrical shearlets provide (nearly) optimally sparse approximations on an idealized class of functions modeling spatio-temportal data and the numerical observation that they provide highly sparse approximations even for more general spatio-temporal image sequences found in dynamic tomography applications. 
To formulate our regularization model, we introduce cylindrical shearlet smoothness spaces, which are instrumental for defining suitable embeddings in functional spaces. We prove that the proposed regularization strategy is well-defined, and the minimization problem has a unique solution (for $ p > 1$). 
Furthermore, we provide convergence rates (in terms of the symmetric Bregman distance) under deterministic and random noise conditions, within the context of statistical inverse learning. We numerically validate our theoretical results using both simulated and measured dynamic tomography data, showing that our approach leads to an efficient and robust reconstruction strategy.
\end{abstract}



\section{Introduction} \label{sec:intro}
Traditional X-ray tomography is typically used to determine the interior structure of an unknown (static) object measuring the total attenuation of X-rays from many orientations. This inverse problem is known to become more severely ill-posed as the number of measurements (i.e., \emph{projections}) decreases. Regularization theory of inverse problems is widely used to overcome errors and limitations caused by very few measurements or too coarse modelling. The underpinning idea is that in order to obtain realistic approximate solutions from limited and noisy data, suitable prior information needs to be built into the problem by means, e.g., of a penalty term, leading to a variational regularization formulation of the reconstruction problem. A wide range of different priors have been proposed to enforce desired structures and properties on the reconstructions. Often, in limited data problems, total variation (TV), which enforces gradient sparsity, is used as a simple, yet effective prior~\cite{Chen2013,Jorgensen2015,Liu2012,Sydky2008,Sydky2006,Tian2011}, as well as its variants, e.g., total generalized variation~\cite{Niu2014}. Alternative strategies have employed wavelets~\cite{Hamalainen2013,Jia2011,Loris2006,Niinimaki2007,rantala2006wavelet,goossens2020}, curvelets~\cite{Candes2000,Frikel13},  shearlets~\cite{Bubba2018,Colonna2010,Guo2013,Riis2018,vandeghinste2013} and combination of those~\cite{goppel2024data} as priors.

However, noise and limited data are not the only cause of  errors in tomographic recon\-structions. In many applications, the target (or parts of it) is dynamic, i.e., it can change between the recording of two consecutive projection images. As a consequence, movement or changes of the target can cause reconstruction artefacts  \cite{milanfar1999model, mutaf2007impact}, unless carefully accounted~\cite{blanke2020inverse} or compensated for~\cite{desbat2007compensation, isola2008motion, hahn2017motion}, either prior~\cite{lu2002tomographic}, during~\cite{van2012combined, niemi2015dynamic, burger2017variational, chen2019new, lunz2021learned, liu2021rethinking} or after the reconstruction process~\cite{ritchie1994predictive, gravier2007tomographic}. 

\textit{Dynamic tomography}, that is, the study of tomographic image  reconstructions of non-stationary targets, is the motivating example of this work.
An extensive review of dynamic imaging problems (including tomography) is provided, e.g., in~\cite{lechleiter2018dynamic, hampel2022review, hauptmann2021image}.
Adopting the classical viewpoint of regularization theory of static inverse problems, we are interested in establishing a variational regularization framework where both spatial structure and temporal evolution of the solution are restricted by some predetermined model or representation system. As compared to those approaches aiming at accurately modelling the uncertainties coming from the non-stationary target in the forward operator, variational approaches turn modelling into a quest for a suitable prior for the penalty term. For example, changes over time can be controlled by penalizing the total variation in time \cite{papoutsellis2021core, pasha2023computational}, possibly coupled with optical flow~\cite{burger2017variational}, separating the dynamic parts from the static background via low-rank constraints \cite{gao2011robust, otazo2015low, wang2014fast, arridge2020joint},
Bayesian estimation \cite{fall2013dynamic, myers2014improving, lan2023spatiotemporal}, 
Kalman filters \cite{schmitt2002efficient, hakkarainen2019undersampled} 
or by sparse multivariate systems \cite{tan2015tensor, Bubba20, bubba2023efficient, kadu2023single}. 

Here we introduce a variational regularization approach based on \textit{cylindrical shearlets}~\cite{Easley21}, a recently proposed direction-aware representation system aimed at efficiently approximating spatio-temporal data, that is, temporal sequences (or videos) of 2D images. Compared to wavelets or even classical shearlets~\cite{shearlets}, the geometry of cylindrical shearlets is better suited to the characteristics of spatio-temporal data, as such data are typically dominated by edge discontinuities in the spatial coordinates only. We are motivated by the recent observation that cylindrical shearlets provide nearly optimally sparse approximations for a class of functions that is useful to model spatio-temporal image sequences and videos~\cite{Easley21, bubba2023efficient}
where they provide $\mathscr{N}$-term approximations with error rate of order $O(\mathscr{N}^{-2})$. In contrast, conventional 3D wavelets and 3D shearlets only achieve approximation rates that are $O(\mathscr{N}^{-\frac{1}{2}})$ \cite{devore1998} and $O(\mathscr{N}^{-1})$ \cite{GL_3D}, respectively.

To formulate our regularization model, we also need to  define the corresponding \textit{cylindrical shearlet smoothness spaces}  (in \cref{ssec:DecompSpace}) by extending the work of one of the authors~\cite{labate13}. This definition is based on the theory of decomposition spaces introduced by K.~Gr\"{o}bner and H. G. Feichtinger~\cite{feichtinger1985banach, feichtinger1987banach} and later adapted into the theory of smoothness spaces by L.~Borup and M.~Nielsen~\cite{borup2007frame}. Building cylindrical shearlet smoothness spaces is instrumental for deriving suitable embeddings in functional spaces. This allows us to express the norm of the target's cylindrical shearlet coefficients as the norm of the target in a suitable cylindrical shearlet smoothness space, similarly to the classical case of wavelet coefficients and Besov norms~\cite{daubechies2004}. 

This sets the foundation for stating our regularised model for dynamic X-ray tomography:
\[
\argmin_{f \in X} \left\{
\frac{1}{2} \ \| \cA f 
    - \gd\|^2_{Y} + \frac{\alpha}{p} \| \sh  (f) \|^p_{\ell^p}
\right\},
\]
where $\cA$ is the time-dependent X-ray transform, $\sh$ the cylindrical shearlet transform, $X$ the cylindrical shearlet smoothness space, $Y$ a Bochner space (integrating over time), and we consider $1<p<2$ (see~\cref{sec:dynTomo} for all the details). This formulation is certainly reminiscent of a classical regularized formulation of static imaging problems. However, there is a twist in our model: while the mismatch term essentially integrates a sequence of separable, static tomographic problems, the regularization term sees the non-stationary target as a spatio-temporal object, tying inextricably the temporal dimension to the spatial ones.  

This approach is seemingly straightforward, both in its formalization (constrained analysis formulation) and conceptualization (treating time as an integral dimension of the target). However, it proves to be very rich to explore and analyze from a theoretical viewpoint and surprisingly effective in its practical numerical implementation. Indeed, we are able to prove that this regularization strategy is well-defined, and the solution of the minimization problem exists and is unique for $p>1$ (see \cref{ssec:regularizingCylindrical}). Furthermore, we provide convergence rates (in terms of the symmetric Bregman distance) under deterministic and random noise conditions (see \cref{ssec:nonsampled}). These results build upon the vast literature on regularization of inverse problems, in particular on the derivation of convergence rates, in various error metrics. The most relevant to our framework are~\cite{Lorenz08}, which considers $p \in [0,2]$, \cite{Grasmair08}  where $p \in [1,2]$,  \cite{Burger04,Grasmair11,Grasmair11bis,Haltmeier12}, which focus on $p = 1$, and~\cite{Hohage2019,Weidling2020} where Tikhonov regularization in Besov spaces is deeply analysed, both in the context of deterministic and random noise.

We remark that the bounds we derive in \cref{prop:conv_rate_full} for the dynamic case are the same of those derived in the literature for the static one.
Finally, these estimates are also analysed under the lens of the statistical inverse learning framework~\cite{blanchard2018optimal} (see \cref{ssec:sampled}), extending to the dynamic setting the work in \cite{bubba2021, Bubba22}, where convergence rates for $p$-homogeneous regularization functional are established, and numerically verified in the context of tomographic imaging with randomized projections. The latter estimates could be relevant for (dynamic) tomography scenarios where the exact projection angles can not be determined a priori by the operator. For example, in really fast tomography, where hardware limits rotation and measurement speed~\cite{ruhlandt2017four}; or \cite{walker2014vivo}, where gating is subject to irregular movement, making the chosen angles irregular too; or in cryo electron tomography where many copies of similar targets are taken at unknown orientation, which can be viewed as having many projections of just one target from unknown directions~\cite{oikonomou2016new}.

To show that the derived theoretical convergence rates are attained in practice, we include several numerical experiments using both simulated and measured dynamic tomographic data (see \cref{sec:numericalTests}). Our  numerical experiments validate the theoretical approximation rates  and show that our theoretical framework leads to a practical and robust reconstruction strategy. While the numerical experiments are restricted to a (2+time)-dimensional setup for computational feasibility, both the theoretical and computational frameworks are extendable to higher dimensions, the only limitation being the computational times. In the spirit of reproducible research, our numerical codes are made available on GitHub~\cite{randDynTomo}.

Finally, we remark that, while the results derived in this work are motivated by the case of dynamic computed tomography, our approach is very general since quite general assumptions are required on the dynamic operator, making it extendable to many dynamic imaging  modality, e.g., magnetic resonance imaging (MRI)~\cite{otazo2015low}, magnetic particle imaging (MPI)~\cite{brandt2024dynamic} or positron emission tomography (PET)~\cite{schmitzer2020dynamic}.

\section{Cylindrical shearlets} \label{sec:cylShearlets}

Cylindrical shearlets were introduced to provide an efficient representation system of spatio-temporal data, assuming that edge discontinuities are mostly located along the spatial coordinates while along the last coordinate axis the geometry is notably simpler. 
We briefly recall below their construction from~\cite{Easley21}.

\subsection{Construction} \label{ssec:cylShConstruction}

Similar to traditional bandlimited shearlets in $L^2(\R^3)$, the construction of cylindrical shearlets consists first in splitting the frequency domain $\R^3$ into symmetric pairs of conical sections, followed by partitioning into rectangular annuli before applying directional filtering. Hence, we first define the cones:
\begin{align*}
    P_1 = \lbrace (\xi_1, \xi_2, \xi_3) \in \R^3 : \abs{\frac{\xi_2}{\xi_1}} \leq 1 \rbrace, \quad P_2 = \lbrace (\xi_1, \xi_2, \xi_3) \in \R^3 : \abs{\frac{\xi_1}{\xi_2}} \leq 1 \rbrace.
\end{align*}
Then, a cone-adapted \emph{cylindrical shearlet system} is a collection of functions
\begin{equation} \label{eq:cShearSys}
    \SHsys^{(\iota)} = \lbrace \psi^{(\iota)}_{\jkm} : j \geq 0, \abs{k} \leq 2^j, \vec{m} \in \Z^3 \rbrace,
\end{equation}
defined separately for the two cones $P_\iota, \iota=1,2$.
Denoting the Fourier transform of $\psi$ by $\hatgen$, the cylindrical shearlet elements of (\ref{eq:cShearSys}) are defined in the Fourier domain as
\begin{equation}    \label{eq:cShearSys2}
    \hatgen^{(\iota)}_{\jkm}(\xi) = \abs{\det A_{(\iota)}}^{-\frac{j}{2}} W(2^{-2j} \xi) V_{(\iota)} \left( B_{(\iota)}^{-k} A_{(\iota)}^{-j} \xi \right) e^{2\pi i B_{(\iota)}^{-k} A_{(\iota)}^{-j} \xi \cdot \vec{m}},
\end{equation}
where the functions $W$ and $V_{(\iota)}$ (defined below) provide the scale- and direction-based filtering and the matrices are
\begin{align}
    A_{(1)} = \begin{pmatrix}
    4 & 0 & 0 \\ 0 & 2 & 0 \\ 0 & 0 & 4
    \end{pmatrix}, \ B_{(1)} = \begin{pmatrix}
    1 & 0 & 0 \\ 1 & 1 & 0 \\ 0 & 0 & 1
    \end{pmatrix}, \ A_{(2)} = \begin{pmatrix}
    2 & 0 & 0 \\ 0 & 4 & 0 \\ 0 & 0 & 4
    \end{pmatrix}, \ B_{(2)} = \begin{pmatrix}
    1 & 1 & 0 \\ 0 & 1 & 0 \\ 0 & 0 & 1
    \end{pmatrix}.
\end{align}
Note that the shear matrices $B_{(\iota)}$ are transposed compared to the usual definitions since we work mostly in the Fourier domain.

To define our Fourier subbands, we start by  choosing $\phi \in L^2(\R^3)$ such that $\bump \in \mathcal{C}_0^\infty$ (i.e., the class of infinitely differentialble functions with compact support), with $0 \leq \bump(\xi) \leq 1$ and
\begin{equation}  \label{eq.phi}
    \bump(\xi) = 1 \ \text{for} \ \xi \in \Big[ -\frac{1}{16}, \frac{1}{16} \Big]^3 \ \text{and} \ \bump(\xi) = 0 \ \text{for} \ \xi \in \R^3 \setminus \Big[ -\frac{1}{8}, \frac{1}{8} \Big]^3.
\end{equation}
Setting
    $W^2(\xi) = \bump^2\Big( \tfrac{\xi}{4} \Big) - \bump^2\big( \xi \big)$,
we obtain the following smooth partition of unity:
\begin{equation} \label{eq:discrete Calderon}
    \bump^2(\xi) + \sum_{j = 0}^\infty W^2\left( 2^{-2j} \xi \right) = 1 \ \text{for} \ \xi \in \R^3,
\end{equation}
where each function $W\left( 2^{-2j} \xi \right)$ is supported on a rectangular annulus $\left[ -2^{2j - 1}, 2^{2j - 1} \right]^3 \, \setminus \, \left[ -2^{2j - 4}, 2^{2j - 4} \right]^3$. This construction partitions the Fourier space into partially overlapping rectangular annuli outside a central cube.
Next, to produce directional filtering, we let
\begin{equation*}
    V_{(1)} (\xi) = v\left( \tfrac{\xi_2}{\xi_1} \right), \; V_{(2)} (\xi) = v\left( \tfrac{\xi_1}{\xi_2} \right),
\end{equation*}
where $v \in \mathcal{C}^\infty(\R)$ is such that $\supp(v) \subset [-1, 1]$ and 
\begin{equation} \label{eq.v}
    \abs{v(z-1)}^2 + \abs{v(z)}^2 + \abs{v(z+1)}^2 = 1, \ \text{for} \ z \in [-1, 1],
\end{equation}
which implies that $\sum_{k \in \Z}\limits \abs{v(z+k)}^2 = 1$, for any $z \in \R$.

By the  definitions of $W$ and $v$, we see that each element $\hatgen_{j,k,\m}^{(1)}$ is contained in the set 
\begin{align*}
    Q^{(1)}_{j,k} = \left\lbrace \xi \in \left[ -2^{2j - 1}, 2^{2j - 1} \right]^3 \, \setminus \, \left[ -2^{2j - 4}, 2^{2j - 4} \right]^3 : \abs{\frac{\xi_2}{\xi_1} - 2^{-j} k} \leq 2^{-j} \right\rbrace.
\end{align*}
The set $Q^{(2)}_{j,k}$ is defined similarly by switching $\xi_1$ and $\xi_2$ around. Notice also that $Q_{j,k} \subset B^k A^j Q_{0,0}$ and for any $j,k$ we have:
\begin{equation*}
\dimFlip Q^{(1)}_{j,k} = Q^{(2)}_{j,k},
\qquad \text{where }\;
\dimFlip = \dimFlip^{-1} = \begin{pmatrix}
    0 & 1 & 0 \\
    1 & 0 & 0 \\
    0 & 0 & 1
    \end{pmatrix}.
\end{equation*}
It is also easy to choose an affine transformation $M_0$ such that $\det(M_0) = 1$ and $\left[ -2^{- 3}, 2^{- 3} \right]^3 \subset M_0 Q_{0,0}^{(1)}$, to include the central low-frequency cube. For example, a simple shift $M_0\xi = \xi + (\frac{1}{3},0,0)$ suffices. Notice that we do not always specify the index of the sets $Q$ and matrices $A$ and $B$ when $\iota=1$ if it is otherwise clear from context.

The wedge-shaped Fourier-domain support of the cylindrical shearlet system is illustrated side-by-side with the partition of traditional 3D shearlets in \cref{fig:3DshearletCones}, showing that cylindrical shearlets tile the Fourier domain differently compared to traditional 3D shearlets.

\begin{figure}[tbh]
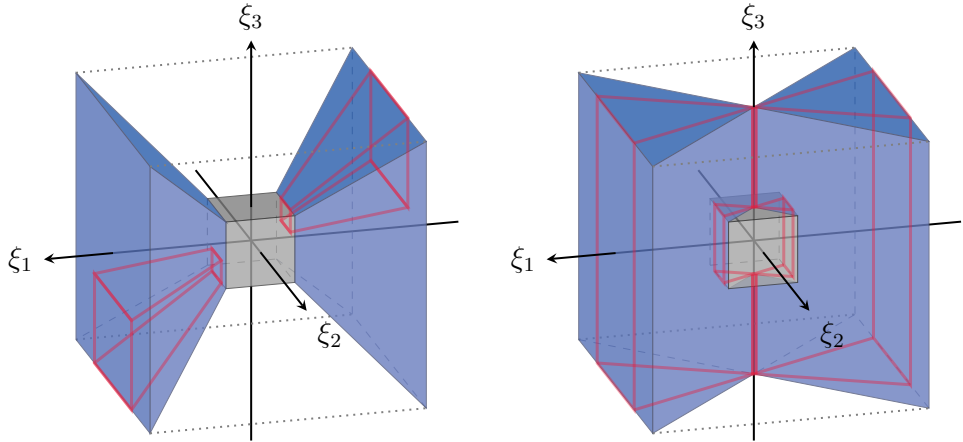

    \addPDFfig{1} 
    \caption{Left: illustration of one pyramid-like section of the traditional 3D shearlet system in the Fourier domain and (in red) the outline of the Fourier support of a shearlet element $\psi^{(\iota)}_{\jkm}$, $\iota=1, j=0, k=(0,1)$ and $\vec{m} \in \Z^3$. Right: illustration of one cone-like section of the 3D cylindrical shearlet system in the Fourier domain and (in red) the outline of the Fourier support of a cylindrical shearlet element $\gen^{(\iota)}_{\jkm}$, $\iota=1, j=0, k=0$ and $\vec{m} \in \Z^3$. In both illustrations, the low frequency cube is shown in gray.}
    \label{fig:3DshearletCones}
\end{figure}

We obtain a smooth Parseval frame of cylindrical shearlets for $L^2(\R^3)$ by combining the two conical systems $\SHsys^{(1)}$ and $\SHsys^{(2)}$, given by~\eqref{eq:cShearSys}, with a coarse scale
system generated by the integer translates of $\phi$, given by~\eqref{eq.phi}. Additionally, to ensure that all elements of this combined system have good decay properties, we modify the elements of the shearlet system overlapping the boundaries of the regions $P_1$ and $P_2$ in the frequency domain. In brief, the shearlet elements from $\SHsys^{(1)}$ and $\SHsys^{(2)}$ overlapping the lines $\xi_1 = \pm \xi_2$ are merged together to ensure a smooth transition between the two pyramidal regions in the Fourier domain. We refer to~\cite{GL_MMNP} for details about the construction of these {\it boundary shearlets.} 	Hence, we obtain  the \textit{discrete cylindrical shearlet system}:
\begin{eqnarray}  
    \SHsys &:=& \lbrace \gen_{-1,\vec{m}} := \phi(x - \vec{m}) \ : \ \vec{m} \in \Z^3 \rbrace
    \cup \lbrace \psi^{(\iota)}_{\jkm} : j \geq 0, \abs{k} < 2^j, \vec{m} \in \Z^3, \iota=1,2 \rbrace \nonumber \\
    && \cup \lbrace \tilde \psi^{(\iota)}_{\jkm} : j \geq 0, k = \pm 2^j, \vec{m} \in \Z^3, \iota=1,2 \rbrace,  \label{eq:cShearSystem}
    \end{eqnarray}  
where the functions $\psi^{(\iota)}_{\jkm}$ are given by~\eqref{eq:cShearSys2} and the functions $\tilde \psi^{(\iota)}_{\jkm}$ are obtained by a simple modification of the functions 
$\psi^{(\iota)}_{\jkm} \in \Psi^{(\iota)}$ (we refer to~\cite{Easley21} for all the details on the construction). We recall the following result.

\begin{theorem}[Theorem~1, \cite{Easley21}] \label{thm:cylShParsevalFrame}
Let $\phi \in L^2(\R^3)$ be chosen such that $\hat \phi \in \mathcal{C}_0^\infty(\R^3)$ and~\eqref{eq.phi} hold, and $v \in C_0^\infty(\R)$ be chosen so that $\supp (v) \subset [-1,1]$ and~\eqref{eq.v} hold.  
Then the discrete cylindrical shearlet system \eqref{eq:cShearSystem} is a Parseval frame
for $L^2(\R^3)$. Furthermore, the elements of this system are $\mathcal{C}^\infty$ and compactly supported in the Fourier domain.
\end{theorem}

With a slight abuse of notation we denote the full discrete cylindrical shearlet parameter group as: 
$$\Lambda = \mcT \times \Z^3 \sim \lbrace (j, k, \iota; \vec{m}) : j \in \N \cup \{0,-1\}, |k| \leqslant 2^j, \iota = 1,2 \, ; \, \vec{m} \in \Z^3 \rbrace.$$
Separating the translations $\vec{m} \in \Z^3$ from the other parameters (in $\mcT$) is needed later in \cref{sssec:cylShDecompSpace}. With this notation,  we define the cylindrical shearlet transform:
\begin{align} \label{eq:cylShTransform}
    \sh: L^2(\R^3) &\longrightarrow \ell^2(\Lambda), \qquad \text{where } \; \sh(f) = \left(\langle f, \psi_\lambda \rangle\right)_{\lambda \in \Lambda}.
\end{align}
In \cref{ssec:DecompSpace}, we show that this function system belongs to a class of decomposition spaces associated with a special partition of the Fourier domain. Let us first motivate why we choose cylindrical shearlets to represent spatio-temporal data.

\subsection{Cylindrical cartoon-like functions and cartoon-like videos} \label{ssec:slantedData}

In this section, we recall the definition of cylindrical cartoon-like functions from \cite{Easley21}, used to model a simple class of spatio-temporal sequences with discontinuities in the spatial plane only. 

Following \cite{donoho2001}, we start by defining a class of indicator functions of sets with $\mathcal{C}^2 $ boundaries. In polar coordinates, let $\rho(\theta): [0, 2\pi) \rightarrow [0, 1]$ be a radius function  and define the set $S = \{x \in \R^2: \abs{x} \leq \rho(\theta), \theta \in [0,2 \pi)\}$.  In particular, the boundary $\partial S$ is given by the curve in $\R^2 $:
\begin{equation}\label{betatheta}
    \beta(\theta) = \begin{pmatrix}
        \rho(\theta) \cos(\theta) \\
        \rho(\theta) \sin(\theta)
    \end{pmatrix},
    \quad \theta \in [0, 2\pi).
\end{equation}
For any fixed constant $Z > 0$,  we say that a set $S \in  {STAR}^2(Z)$ if $S \subset [0, 1]^2$ and $S$ is a translate of a set satisfying \eqref{betatheta} where 
\begin{equation}\label{sup:cond}
\sup \abs{\rho^{''}(\theta)} \leq Z,\quad \rho \leq \rho_{0} < 1.
\end{equation}
Hence, denoting with $\mathcal{C}_0^2 ([0, 1]^2)$ the class of twice differentiable functions with support in $[0, 1]^2$, the set of {\bf cartoon-like images} $\mathcal{E}^2(Z) \subset L^2(\R^2)$ was defined in~\cite{donoho2001} as the
set of functions 
$h = h_0 \, \chi_S$
where
$h_0 \in \mathcal{C}_0^ 2 ([0, 1]^2 )$, $S \in  {STAR}^2(Z)$ and $\sum_{\abs{\alpha} \leq 2} \norm{D^{\alpha} h}_{\infty}\leq 1.$  

Using the above notation and definition, some of the authors  defined in \cite{Easley21} the class of {\bf cylindrical cartoon-like functions} as follows.
\begin{definition}[{\cite[Equation 12]{Easley21}}]  \label{def.functions}
Given a constant $Z > 0$, the class of {\bf cylindrical cartoon-like functions} is the set of functions $\mathcal{U}^2(Z) \subset L^2(\R^3)$ of the form
\begin{equation}   \label{def.cci}
  f(x) = f(x_1,x_2,x_3) =h(x_1,x_2)\, \chi_{S}(x_1,x_2)\, z(x_3),  
\end{equation}
where $S \in STAR^2(Z)$, $h \in  \mathcal{C}_0^2 ([0, 1]^2 )$, $z \in \mathcal{C}_0^2 ([0, 1])$ and 
$\norm{D^{\alpha} h}_{\infty}, \norm{D^{\alpha} z}_{\infty}\leq 1.$ 
\end{definition}
Note that for any cylindrical cartoon-like function $f$ and constant $c \in \R$, the function $\tilde f_c(x) = f(x_1,x_2,c)$ is a cartoon-like image, where the edge boundary $\partial S$ is the same for all $x_3$.

\subsubsection{Sparse approximations}
Cylindrical shearlets were shown to provide (nearly) optimally sparse approximations, in a precise sense, for the class of cylindrical cartoon-like functions defined by~\cref{def.cci}. 

Let  $\SHsys = \{\psi_\lambda: \lambda \in \Lambda\}$  be the Parseval frame of discrete cylindrical shearlets given by \eqref{eq:cShearSystem}. For any function $f \in L^2(\R^3)$, 
the \emph{cylindrical shearlet coefficients} are the elements of the sequence $\sh (f) = \{\ip{f}{\psi_{\lambda}}:\lambda \in \Lambda \}$. We denote by $f_\mathscr{N}$ the $\mathscr{N}$-term approximation of $f$  obtained from the $\mathscr{N}$ largest coefficients of its cylindrical shearlet expansion, namely:
$$f_\mathscr{N}=\sum_{\lambda \in \mathcal{I}_\mathscr{N}} \ip{f}{\psi_{\lambda}} \, \psi_{\lambda},$$
where $\mathcal{I}_\mathscr{N} \subset \Lambda$ is the set of indices corresponding to the $\mathscr{N}$ largest entries of the absolute-value sequence $\{|\ip{f}{\psi_{\lambda}}|:\lambda \in \Lambda \}$. We have the following result from \cite{Easley21}.

\begin{theorem}[{\cite[Theorem 3]{Easley21}}]  \label{th:main}
Given any $Z>0$, let $f \in \mathcal{U}^2(Z)$, the set of cylindrical cartoon-like functions, and $f_\mathscr{N}$ be the $\mathscr{N}$-term approximation to $f$ defined above. Then, for $\mathscr{N} \in \N$ there exists a constant $c$ independent of $\lambda$ in $\sh (f)$ and $\mathscr{N}$ such that:
$$\norm{f-f_\mathscr{N}}^2_{L^2} \leq c \, \mathscr{N}^{-2} {(\ln(\mathscr{N}))}^{3}.$$
\end{theorem}

\begin{remark}
The decay rate in \cref{th:main} is nearly optimal, in the sense that, up to a logarithmic factor, no other representation system satisfying polynomial depth search can achieve an asymptotic approximation rate faster than $\mathscr{N}^{-2}$ for the class of functions $\mathcal{U}^2(Z)$ as shown in~\cite{Easley21}. By contrast, separable 3D wavelets and conventional 3D shearlets, up to a logarithmic factor, only achieve approximation rates that are $O(\mathscr{N}^{-1/2})$ \cite{devore1998} and $O(\mathscr{N}^{-1})$ \cite{GL_3D}, respectively. The notion of representation systems {\it satisfying polynomial depth search} includes any Parseval frame and was introduced to include function representations that are more general than orthonormal bases but not as large as to make the approximation problem computationally intractable \cite{candes1999,GL_3D}. The significance of Theorem~\ref{th:main} is that sparsity-promoting \emph{cylindrical shearlet norms are expected to provide a sparsity-inducing prior for data in the class of cylindrical cartoon-like functions}. 
\end{remark}

\subsubsection{Cartoon-like videos} \label{ssec:cartoon-videos}
Cylindrical cartoon-like functions are a rather coarse model of spatio-temporal data, as they exhibit the same discontinuous edge at the same location over the entire temporal sequence. We introduce here a generalization of this class of functions providing a more realistic model of spatio-temporal data found in applications.

\begin{definition}  \label{def.videos}
Given a constant $Z > 0$, we define the class of {\bf cartoon-like videos} as the functions $\videos \subset L^2(\R^3)$ of the form
\begin{eqnarray}\label{slant.funct}
    f(x) = h(x_1 - q_1(x_3), x_2 - q_2(x_3)) \, \chi_{S}(x_1 - q_1(x_3), x_2 - q_2(x_3)) \, z(x_3),
\end{eqnarray}
 where $S \in STAR^2(Z)$, $h \in  \mathcal{C}_0^2 ([0, 1]^2 )$, $z \in \mathcal{C}_0^2 ([0, 1])$, 
$\norm{D^{\alpha} h}_{\infty}, \norm{D^{\alpha} z}_{\infty}\leq 1$  and  $q_1, q_2 \in \mathcal{C}_0^2 ([0, 1])$ with $q_1(0)=q_2(0)=0.$   
\end{definition}

By construction, for any $x_3$, discontinuities in cartoon-like videos occur on the boundary of the  sets:
$$S_{x_3}=\{(x_1,x_2) \in \R^2:\, |(x_1 - q_1(x_3), x_2 - q_2(x_3))| \leq \rho(\theta), \theta \in [0,2 \pi)\},$$
where the radius function $\rho(\theta)$ is given above and $S=S_0$; i.e., we have the boundary curves: 
$$\theta \mapsto \left( \rho(\theta)\cos(\theta) + q_1(x_3),\rho(\theta)\sin(\theta) + q_2(x_3) \right).$$
Hence, we can write \eqref{slant.funct} as:
$$ f(x) = h(x_1 - q_1(x_3), x_2 - q_2(x_3)) \, \chi_{S_{x_3}}(x_1,x_2) \, z(x_3). $$
\begin{figure}[ht]
    \centering
    \begin{tikzpicture}
        \begin{scope}[xshift=-3cm] 
            \node[anchor=center] at (0,0) {\includegraphics[scale=1.5]{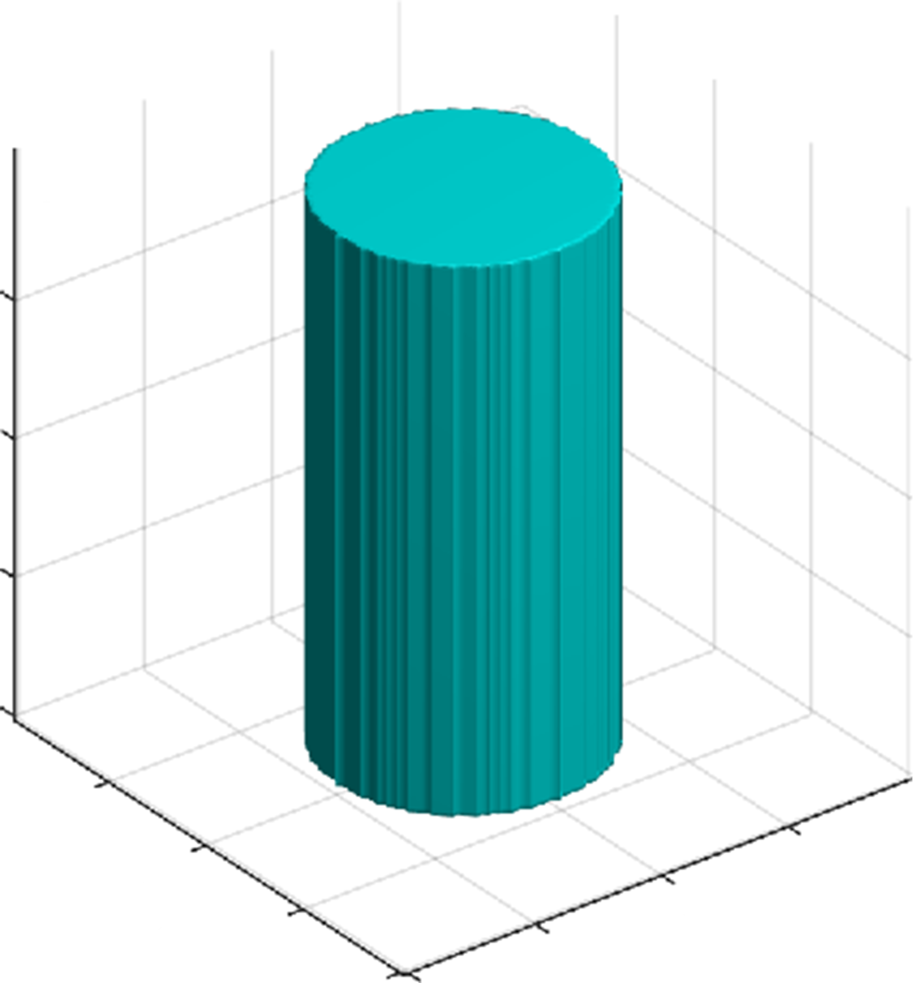}};
            \node at (-2.55,0.3) {\large $x_3$};
            \node at (-1.5,-2) {\large $x_2$};
            \node at (1.3,-2.15) {\large $x_1$};
        \end{scope}
        \begin{scope}[xshift=3cm] 
            \node[anchor=center] at (0,0) {\includegraphics[scale=1.5]{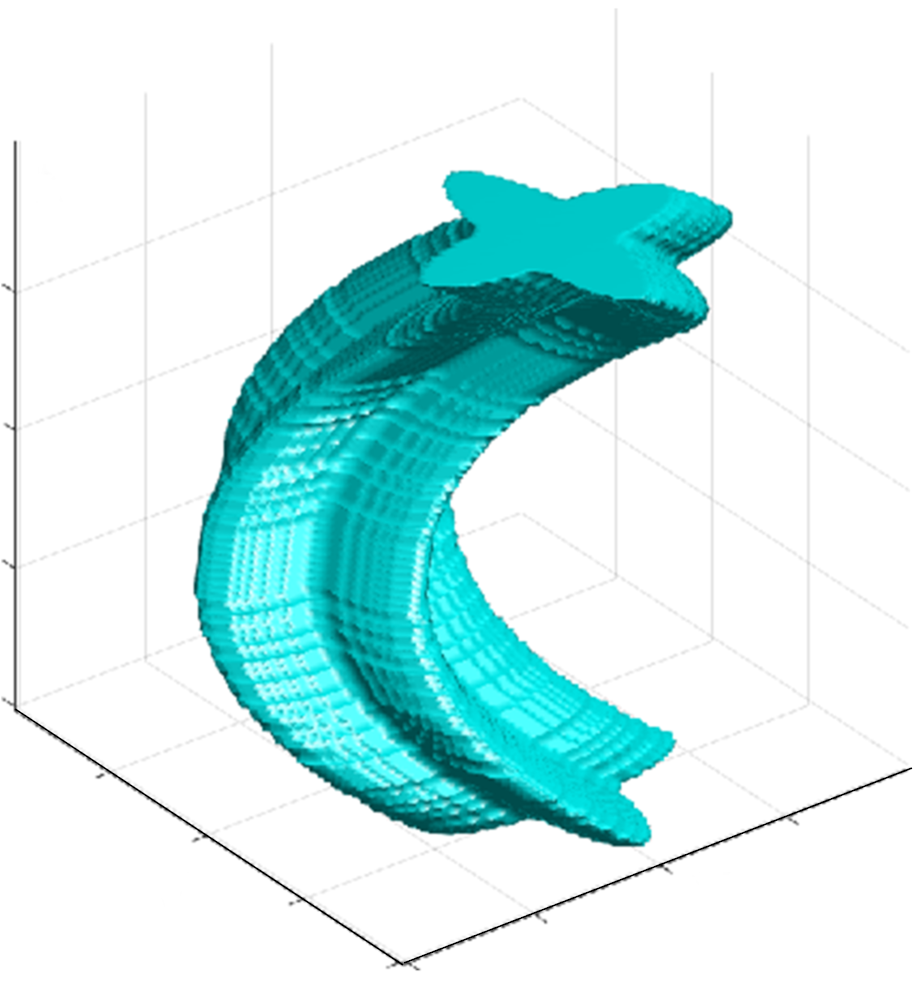}};
            \node at (-2.55,0.3) {\large $x_3$};
            \node at (-1.5,-2) {\large $x_2$};
            \node at (1.3,-2.15) {\large $x_1$};
        \end{scope}
    \end{tikzpicture}
 \caption{Example of a cylindrical cartoon-like function (left) and a cartoon-like video (right)  in $L^2(\R^3)$. }
 \label{fig.polartwo}
\end{figure}
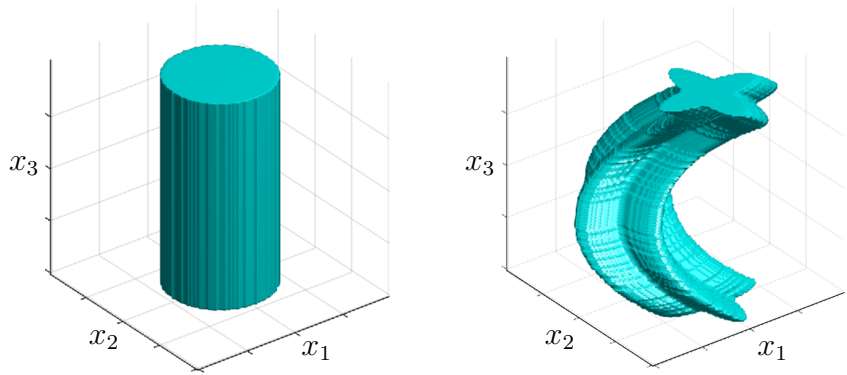

\Cref{fig.polartwo} shows an example of a cylindrical cartoon-like function (left) and an example of a cartoon-like video (right) defined by displacing a starshape region $S \subset \R^2$ as a function of the
$x_3$ coordinate. Note that the region $S$ is allowed to gradually disappear if $z(x_3) = 0$ at some $x_3$.

\begin{remark}
Based on our numerical experiments reported below, we conjecture that the sparse approximation result in \cref{th:main} is also valid for the class of cartoon-like videos. However, no proof of this conjecture is available at this time.  
\end{remark}

\subsection{Decomposition and smoothness spaces} \label{ssec:DecompSpace}

We introduce the class of \emph{cylindrical shearlet smoothness spaces} by adapting the construction of the shearlet decomposition spaces \cite{labate13} associated with traditional shearlets \cite{shearlets, guo10} to cylindrical shearlet systems. 
We start by recalling the basic definitions of the decomposition space theory from \cite{feichtinger1985banach, feichtinger1987banach, borup2007frame}. For brevity, we write $f \lesssim g$ if there exists a constant $C > 0$ independent of $f$ and $g$ such that $f \leq C g$. We denote $f \lesssim g \lesssim f$ by $f \simeq g$.

\subsubsection{Constructing decomposition spaces} \label{sssec:DecompSpaceIntro}

An \emph{admissible covering} $\lbrace Q_i \subset \R^d : i \in I \rbrace$ is any collection of bounded and measurable sets such that $\bigcup_{i \in I} Q_i = \R^d$ and there exists a uniform bound on how many of the sets $Q_i$ can overlap at any point in $\R^d$. 
Given any admissible covering a \emph{bounded admissible partition of unity} (BAPU) is a family of functions $\lbrace \gamma_i : i \in I \rbrace$ such that, each $\gamma_i$ is supported only in $Q_i$; $\sum_i \gamma_i(\xi) = 1$ for any $\xi \in \R^d$; and $\sup_i |Q_i|^{\frac{1}{p}-1} \| \widecheck{\gamma}_i \|_{L^p} < \infty$ for any $p \in (0, \infty]$, where $\widecheck{\gamma}$ is the inverse Fourier transform of $\gamma$.

We are interested in more structured coverings. Let $\mcT = \lbrace \msT_k: x \mapsto A_k x + b_k : k \in \N \rbrace$ be a collection of invertible affine transformations on $\R^d$ and, to ease the notation, we omit the index $k$, if possible. Consider two bounded open sets $P, Q \subset \R^d$ such that $P$ is compactly contained in $Q$ and the collections $\mcQ = \lbrace \msT Q : \msT \in \mcT \rbrace$ and 
$\mcP = \lbrace \msT P : \msT \in \mcT \rbrace$ 
are both admissible coverings. If, in addition there is a constant $C$ such that:
\begin{align*}
    (A_k Q + b_k) \cap (A_l Q + b_l) \neq \emptyset \quad \Longrightarrow \quad \|A_l^{-1} A_k \|_{\ell^\infty} < C,
\end{align*} 
then we call $\mcQ$ a \emph{structured admissible covering} and $\mcT$ a \emph{structured family of affine transformations}. Then, we know that $\mcQ$ and $\mcT$ are also associated with a \emph{squared BAPU} $\lbrace \bapu_\msT : \msT \in \mcT \rbrace$, i.e., the second condition is replaced by $\sum_\msT |\bapu_\msT (\xi) |^2 = 1$ and the third one only needs to hold for any $p \in (0,1]$. 

Denote $|\msT| = |\det A|$ and let $\mcQ$ and $\mcT$ be as defined above. Suppose $K_a = [-a, a]^d$, for $a>0$, is a cube in $\R^d$ containing $Q$. Then we can define the system: 
\begin{align*}
    \lbrace 
    \eta_{\vec{m},\msT} \; : \;
    \widehat{\eta}_{\vec{m},\msT} =  \bapu_\msT e_{\vec{m},\msT},  \ \vec{m} \in \Z^d, \msT \in \mcT \rbrace,
\end{align*}
where $\bapu_\msT$ is a squared BAPU and
\begin{align*}
    e_{\vec{m},\msT}(\xi) = \frac{\chi_{K_a} (\msT^{-1} \xi) e^{i \frac{\pi}{a} \vec{m} \cdot \msT^{-1}\xi}}{\sqrt{|2a|^d |\msT|}}, \ \vec{m} \in \Z^d, \msT \in \mcT.
\end{align*}
It follows that the system of $\eta_{\vec{m},\msT}$'s is a Parseval frame of $L^2(\R^d)$~\cite{borup2007frame}. Then, if we denote:
\begin{align*}
    \eta_{\vec{m},\msT}^{(p)} = |\msT|^{\frac{1}{2} - \frac{1}{p}} \eta_{\vec{m},\msT},
\end{align*}
we can characterize all of the decomposition spaces $D(\mcQ, L^p, \mcY_w)$ using the associated frame coefficients $\lbrace \< f, \eta_{\vec{m},\msT}^{(p)} \> : \vec{m}, \msT \rbrace$ in the way specified by the following result.

\begin{proposition}[{\cite[Proposition 3]{borup2007frame}}]
Let $\mcQ = \lbrace \msT Q : \msT \in \mcT \rbrace$ be a structured admissible covering, $\mcY_w$ a solid (quasi-)Banach sequence space on $\mcT$ with a $\mcQ$-moderate weight%
\footnote{A strictly positive function $w$ is $\mcQ$-moderate if exists $C > 0$ such that $w(x) \leqslant C w(y)$ for all $x,y \in Q_i \in \mcQ$ and all $i \in I$. A $\mcQ$-moderate \textit{weight} on $I$ is derived as the sequence $\left( w_i = w(x_i) \right)_{i \in I}$, where each $x_i \in Q_i$.}  $w$. %
For $0 < p \leqslant \infty$ we have the characterization:
\begin{align*}
    \| f \|_{D(\mcQ, L^p, \mcY_w)} \simeq \Bigg\| \bigg( \Big( \sum_{\vec{m} \in \Z^d} |\< f, \eta_{\vec{m},\msT}^{(p)} \> |^p \Big)^{\frac{1}{p}} \bigg)_{\msT \in \mcT} \Bigg\|_{\mcY_w}.
\end{align*}
\end{proposition}
This indicates that there is a coefficients space $d(\mcQ, \ell^p, \mcY_w)$ associated with the decomposition space and we can define it as the set of coefficients $c = (c_{\vec{m},\msT})_{\vec{m} \in \Z^d,\msT \in \mcT}$ which satisfy:
\begin{align*}
    \| c \|_{d(\mcQ, \ell^p, \mcY_w)} = \Bigg\| \bigg( \Big( \sum_{\vec{m} \in \Z^d} | c_{\vec{m},\msT} |^p \Big)^{\frac{1}{p}} \bigg)_{\msT \in \mcT} \Bigg\|_{\mcY_w} < \infty.
\end{align*}
If we define the \emph{coefficient operator} $\msK: D(\mcQ, L^p, \mcY_w) \rightarrow d(\mcQ, \ell^p, \mcY_w)$ and the \emph{reconstruction operator} $\msK^*: d(\mcQ, \ell^p, \mcY_w) \rightarrow D(\mcQ, L^p, \mcY_w)$ in the obvious way:
\begin{equation*}
    \msK (f) = \left( \< f, \eta_{\vec{m},\msT}^{(p)} \> \right)_{\vec{m},\msT}, 
    \qquad \text{and} \qquad 
    \msK^* (c) = \sum_{\vec{m},\msT} c_{\vec{m},\msT} \, \eta_{\vec{m},\msT}^{(p)},
\end{equation*}
then both operators are bounded and $\msK^* \msK = \Id_{D(\mcQ, L^p, \mcY_w)}$ \cite{borup2007frame}. From now on we shall use the more familiar names \emph{analysis} and \emph{synthesis operator} when referring to the coefficient $\msK$ and reconstruction $\msK^*$ operators, respectively.\!
\footnote{Historically it seems that the names coefficient and reconstruction operator were used in the theory of atomic decompositions, whereas analysis and synthesis originate from the more classical frame theory of Hilbert spaces. The theory of Banach frames lies at their intersection (see \cite{casazza1999frames, carando2011reconstruction} for more details and context).}

Finally, as shown in \cite{borup2007frame}, by choosing $\mcT$ to be a structured family of affine transformations (where $\msT x := A_\msT x + b_\msT$), $\mcY_w = \ell^q_{w^{\beta}}(\mcT)$, with $w$ a $\mcQ$-moderate function, $\beta \in \R$ and the weights defined by $w^{\beta} = \left( w^\beta(\msT) \right)_{\msT \in \mcT}$, the decomposition spaces form certain \textit{smoothness spaces}, which we denote by $S^\beta_{p,q}(\mcT, w)$. 
For example, with wavelet transform-like dilations $\mcT$ we can identify the smoothness space $S^\beta_{p,q}$ with a suitable Besov space \cite{labate13}. 

In the next section, we construct a similar setting for the cylindrical shearlets.

\subsubsection{Cylindrical shearlet decomposition spaces} \label{sssec:cylShDecompSpace}

We fix the structured admissible covering by using exactly those affine transformations defined for the cylindrical shearlet system \cref{eq:cShearSys} in \cref{ssec:cylShConstruction} and write:
$$ \mcT := \lbrace \msT = \dimFlip^{\iota-1} B^k A^j \, : \, \iota = 1,2, |k| \leq 2^j, j \in \N \rbrace \cup \lbrace M_0 \rbrace \qquad \text{and} \qquad
\mcQ := \lbrace \msT Q_{0,0} \, : \, \msT \in \mcT \rbrace. $$

Now, a squared BAPU can be obtained by cancelling out the normalization term in the cylindrical shearlet system (\ref{eq:cShearSystem}), that is:
\begin{align} \label{eq:cylShBAPU}
    \bapu_{\msT_{j,k,\iota}}(\xi) &= \abs{\det A_{(\iota)}}^{\frac{j}{2}} \hatgen^{(\iota)}_{j,k,0}(\xi) := |\msT|^{\frac{1}{2}} \hatgen_{\msT,0}(\xi).
\end{align}
Hence we define the cylindrical shearlet smoothness space $S^\beta_{p,q}(\R^3)$, for $\beta \in \R, 0 < p,q \leqslant \infty$ and $w(\msT) = 2^j$, as the set of functions $f \in \mcS'(\R^3)$ such that:
\begin{align}
    \| f \|_{S^\beta_{p,q}} &\simeq \Big( \sum_{\msT \in \mcT} \Big( 2^{\beta j} \sum_{\vec{m} \in \Z^3} \big| \< \widehat{f}, |\msT|^{\frac{1}{2} - \frac{1}{p}} \hatgen_{\msT, \vec{m}} \> \big|^p \Big)^{\frac{q}{p}} \Big)^{\frac{1}{q}} \nonumber \\
    &= \Big( \sum_{\msT \in \mcT} \Big( 2^{\beta j} |\msT|^{\frac{p}{2} - 1} \sum_{\vec{m} \in \Z^3} \big| \< \widehat{f}, \hatgen_{\msT, \vec{m}} \> \big|^p \Big)^{\frac{q}{p}} \Big)^{\frac{1}{q}} < \infty.
\end{align}
In particular for $q = p$ and if we denote $\tilde{w} =\tilde{w}_{\beta,p}(\msT) :=  w^\beta (\msT) \, |\msT|^{\frac{p}{2}-1}$ we can define the cylindrical shearlet smoothness space norm:
\begin{align}
    \| f \|_{S^\beta_{p,p}} &\simeq \| \sh (f) \|_{\ell^p(\tilde{w})} = \Big( \sum_{\lambda \in \Lambda} \tilde{w}(\lambda)| \langle f, \psi_\lambda \rangle |^p \Big)^{\frac{1}{p}} <  \infty \quad \forall f \in S^\beta_{p,p},
\end{align}
where the cylindrical shearlet transform $\sh$ from~\eqref{eq:cylShTransform} acts as the analysis operator. Essentially, we expressed the (weighted) $\ell^p$-norm of the cylindrical shearlet coefficients
of $f$ as the norm of $f$ in the cylindrical shearlet smoothness space $S^\beta_{p,p}$.
Moreover, since the cylindrical shearlet system~\eqref{eq:cShearSystem} is a (smooth) Parseval frame (\cref{thm:cylShParsevalFrame}), we have the following result.
\begin{proposition}\label{prop:coercivity}
Let $0 < p \leq 2$ and $\| f \|_{L^2} \geq 1$. For any weight $\tilde{w}$ such that $\inf_{\lambda \in \Lambda} \tilde{w}(\lambda) > 0$, there is $C > 0$ such that
\begin{align} \label{eq:ell(w)bound}
    \|f\|_{L^2} \leq \| \sh (f) \|_{\ell^p} \leq C \| \sh (f) \|_{\ell^p(\tilde{w})}.
\end{align}
In particular, the cylindrical shearlet smoothness space norm is coercive in $L^2(\R^{3})$. 
\end{proposition}

\begin{proof}
First, we use the fact that for $0 < p \leq 2$ 
we have:
\begin{align} \label{eq:l2<lp}
   \| u \|_{\ell^2}^2 \leq \sum_{i = 1}^\infty |u_i|^p = \| u \|_{\ell^p}^p \qquad \forall u \in \ell^2: \|u\|_{\ell^2} \leq 1,
\end{align}
which follows since $|u_i|^2 \leq |u_i|^p \leq 1$ for all $i$. 
Next, notice that for any $\|v\|_{\ell^2} > 1$ we have:
\begin{align}
    1 < \|v\|_{\ell^2}^2 = \|v\|_{\ell^2}^2 \| \tfrac{v}{\|v\|_{\ell^2}} \|_{\ell^2}^2 \leq \|v\|_{\ell^2}^2 \|\tfrac{v}{\|v\|_{\ell^2}} \|_{\ell^p}^p = \|v\|_{\ell^2}^{2-p} \|v \|_{\ell^p}^p \qquad
    &\Longleftrightarrow \quad \|v\|_{\ell^2}^p \leq \|v\|_{\ell^p}^p \nonumber \\
    &\Longleftrightarrow \quad \|v\|_{\ell^2}  \leq \|v\|_{\ell^p}. \label{eq:ell-ineq}
\end{align}
Since cylindrical shearlets form a Parseval frame, we get the unweighted version of~\eqref{eq:ell(w)bound} by using \eqref{eq:ell-ineq}, i.e., $\|f\|_{L^2} = \|\sh (f) \|_{\ell^2} \leq \|\sh (f) \|_{\ell^p}$. 
Then, given any (not necessarily $\mcQ$-moderate) weight $\tilde{w}$, and denoting $c = \inf_{\lambda \in \Lambda} \tilde{w}(\lambda)$, it follows that:
\begin{align*}
    \| f \|_{L^2} \leq \frac{1}{c^{1/p}} c^{1/p} \| \sh (f) \|_{\ell^p} \leq \frac{1}{c^{1/p}} \|\sh (f) \|_{\ell^p(\tilde{w})}.
\end{align*}
This proves the coercivity and inequality~(\ref{eq:ell(w)bound}) for $C = \frac{1}{c^{1/p}}$.
\end{proof}

Since in \cref{sec:dynTomo} we are interested in weights of the form $\tilde{w}(\msT) = w^\beta(\msT) \, |\msT|^{\frac{p}{2}-1}$ with $w(\msT) = 2^j$, we conclude this section proving the next result.

\begin{corollary} \label{cor:betaBound}
For $\tilde{w} = 2^{j\beta} \, |\msT|^{\frac{p}{2}-1}$, if $\beta \geq \frac{5}{4}(2-p)$, then $\|\sh (f) \|_{\ell^p(\tilde{w})}$ is coercive.
\end{corollary}
\begin{proof}
We need to check when $\inf \tilde{w} > 0$. Since $\det(M_0) = 1$ and $\det(\dimFlip) = 1$, we can focus on the case when $j \geq 1$. Hence the sequence of weights is always monotone 
$\tilde{w}(j) = 2^{\beta j} \, 2^{\frac{5}{2}j \left( \frac{p}{2}-1 \right)} = 2^{j\beta + \frac{5}{2} j \left( \frac{p}{2}-1 \right)}$. It is clear that $\inf_{j \geq 0}\tilde{w}(j) = \lim_{j \to \infty}\tilde{w}(j) = 0$ only when the exponent is negative. Therefore, the necessary condition for $\beta$ is given by:
\begin{align*}
j\beta &+ \frac{5}{2} j \left( \frac{p}{2}-1 \right) \geq 0 \quad
    \Longleftrightarrow \quad \beta \geq \frac{5}{4}(2-p).
\end{align*}
\end{proof}

\begin{remark} \label{rmrk:noWeights}
The usual choice for unweighted $\ell^p$ regularization of the (cylindrical) shearlet coefficients actually corresponds to the limit case $\beta = \frac{5}{4}(2-p)$, i.e., $w^\beta (j) = 2^{\frac{5}{4}(2-p)j}$ in order to have $\tilde{w}(\msT) = 1$ for all $\msT \in \mcT$.
\end{remark}

For the remainder of the paper, we only consider the case $p > 1$ to carry out our theoretical analysis. When $p>1$, the smoothness spaces $S^\beta_{p,p}(\R^3)$ are reflexive Banach spaces, since the spaces $L^p(\R^3)$ and $\mcY_w = \ell^p_{w^{\beta}}(\mcT)$ are both reflexive Banach spaces (cf. \cite[Theorem 3.21]{voigtlaender2023embeddings} and \cite[Theorem 2.8]{feichtinger1985banach}). 
Note also that, by inequality~\eqref{eq:l2<lp} and the fact that the cylindrical shearlet system is a Parseval frame, it follows that convergence in ${S^\beta_{p,p}}$ implies convergence in $L^2$.

\section{Regularizing a dynamic inverse problem with cylindrical shearlets} \label{sec:dynTomo}

In this section, we formulate the inverse problem we want to tackle using cylindrical shearlet regularization.  
To this end, we first introduce the related notation and the necessary assumptions, along with a motivating example (i.e., the inverse problem arising from dynamic Computed Tomography (CT)) which will be used in the numerical simulations in \cref{sec:numericalTests}. 
Then we derive convergence rates considering two scenarios: full measurements and (randomly) subsampled measurements.

\subsection{Dynamic inverse problems}
\label{ssec:dynInvProb}
Consider the following (dynamic) inverse problem:
\begin{equation}\label{eq:ProblemSetting}
\text{given } \quad \gd \in Y \ : \quad \gd(t) = (\cA f^{\dag}) (t) + \delta \epsilon(t) \quad \text{a.e.}\ t \in (0,T), 
\qquad \text{recover } \quad  f^{\dag} \in X.
\end{equation}
We can formulate the problem in a rather general framework assuming that $X$ is a reflexive Banach space, $X \subset L^2(\Omega \times (0,T))$, being $\Omega \subset \R^2$ a bounded set, 
 and $Y = L^2(0,T; \tilde{Y})$, where $\tilde{Y}$ is a Hilbert space. The notation $L^2(0,T; \tilde{Y})$ is used for the Bochner space:
\[
L^2(0,T; \tilde{Y}) = \left\{  g \colon [0,T]  \rightarrow \tilde{Y} \text{ measurable s.t. } \int_0^T \|g(t)\|_{\tilde{Y}}^2\ dt < \infty\right\}.
\]
For later purposes (see \cref{ssec:sampled}), it is convenient to interpret the elements of $\tilde{Y}$ as functions, whose variable lives in the compact set $\mathcal{U}$ and whose values range in a suitable Hilbert space $\mathcal{V}$: i.e.,  $\tilde{Y} = L^2(\mathcal{U};\mathcal{V})$.
We assume that $\cA$ is a bounded linear operator from $X$ to $Y$. 

In the case of our motivating example (i.e., dynamic CT), the operator $\cA$ is the Radon transform, which we briefly introduce here. This definition is the first step towards introducing our final model, which is detailed in the following subsections.

\begin{MotExample}[Radon transform]
\label{motex:RadonTransform}

For a function $f \in L^2(\Omega)$, the Radon transform $\mathcal{R} f$ is defined as follows:
\[
\mathcal{R}f (\theta,s) = \int_{\R} f(s \theta + \sigma \theta^\perp) d\sigma \qquad \theta \in S^1, s \in \R.
\]
The function $\mathcal{R}f$ (the so-called \textit{sinogram}) belongs to the space $H^{1/2}([0,2\pi)\times(-\bar{s},\bar{s}))$ for a suitable $\bar{s} > 0$ depending on $\Omega$ (see \cite[Theorem 2.10]{natterer2001mathematical}). In particular, the range of $\mathcal{R}$ is a subset of $L^2([0,2\pi)\times(-\bar{s},\bar{s}))$, which we may interpret as $L^2(\mathcal{U};\mathcal{V})$ being $\mathcal{U}=[0,2\pi)$ and $\mathcal{V} = L^2(-\bar{s},\bar{s})$.
The dynamic operator $\cA \colon L^2(\Omega \times (0,T))\rightarrow Y$ associated with the Radon transform is then defined by $(\cA f)(t) = \mathcal{R}(f(\cdot,t))$ for almost every time $t$.

\end{MotExample}

Next, going back to problem~\eqref{eq:ProblemSetting}, for the error term $\delta \epsilon$, with $\delta > 0$, we consider two different scenarios, for which we formulate different assumptions:
\begin{itemize}
    \item \textit{Deterministic} noise. In this case, $\epsilon \in Y$ is a generic (unknown) function perturbing the measurement, for which we only assume that
    \begin{equation}
     \| \epsilon \|_Y \leq 1.
        \label{eq:noise_det}
    \end{equation}
    This is the most common environment in which classical regularization techniques for inverse problems have been introduced and developed (see~\cite{engl1996regularization}).
    \item \textit{Statistical} noise. In this scenario, the measurement acquisition is modeled as a random experiment, and the uncertainty $\epsilon$ is modeled as a random process in time such that, at almost every $t$, $\epsilon(t)$ is a random variable taking values in $\tilde{Y}$. 
    In this case, analogously to \cite{bubba2021}, we assume that at almost every time $t$ the variable $\| \epsilon(t)\|_{\tilde{Y}}$ 
    is sub-Gaussian (see \cite[Definition 2.5.6]{vershynin2018high}) 
    and
    \begin{equation}
    \E[\epsilon(t)] = 0 \quad \text{and} \quad \big\| \| \epsilon(t) \|_{\tilde{Y}}\big\|_{\text{sG}}  \leq 1 \qquad \text{a.e.}\ t \in (0,T).
        \label{eq:noise_stat}
    \end{equation}

\end{itemize}

It is clear that the deterministic case can be interpreted as a single realization of the statistical one, provided that the noise distribution is bounded. In  \cref{ssec:regularizingCylindrical} and \cref{ssec:nonsampled}, our analysis will consider both scenarios, whereas in  \cref{ssec:sampled} we will restrict only to the statistical one in the case of (randomly) sampled measurements. 
We point out that, in the statistical noise scenario, the assumption that $\epsilon(t)$ is a random variable on $\tilde{Y}$ can be restrictive: for example, as shown in \cite{kekkonen2014analysis}, Gaussian white noise cannot be interpreted as a random variable on the infinite-dimensional space $L^2(\Omega)$. In order to encompass a larger family of noise models, the reader is referred to the approach presented in \cite{burger2018large}.

\subsection{Regularization with cylindrical shearlets}
\label{ssec:regularizingCylindrical}
In this section, we show how cylindrical shearlets can be employed to define regularization strategies for dynamic inverse problems. Our inspiration comes from applications in the non-dynamic setting~\cite{chambolle1998nonlinear, rantala2006wavelet}:
indeed, in many static, imaging problems 
it is common practice to assume \textit{a priori} that the target is sparse 
with respect to some wavelet representation. This is modeled as a variational problem in a suitable Besov space (i.e., employing Besov norms as a regularizer).

In our setup, motivated by the observation that cylindrical cartoon-like functions (i.e., ``dynamic images'') are optimally sparse with respect to cylindrical shearlet frames (see \cref{ssec:slantedData}), we consider the smoothness spaces $S_{p,p}^\beta(\R^3)$ as the natural 
embedding space for this class of problems. 

In principle, the theoretical analysis in \cref{sec:cylShearlets} would hold for any $p>0$, but we are only interested here in the case where $p<2$ (to enforce sparsity) and $p>1$ (to leverage strong convexity).
Notice that in place of a smoothness space as the ones introduced in \cref{sssec:cylShDecompSpace} we can consider its subspace consisting of those functions which are supported on the compact set $\overline{\Omega}\times [0,T]$. 
We therefore set
\begin{equation}
    X = S_{p,p}^\beta(\R^3) \cap \lbrace  f: \supp(f) \subset \overline{\Omega} \times [0, T] \rbrace, \label{eq:set_X}
\end{equation}
with $1<p<2$ and $\beta \geq \frac{5}{4}(2-p)$. 
As remarked above, any converging sequence in $X$ must converge also in $L^2$-norm implying that the limit must also be supported on the compact set $\overline{\Omega}\times [0,T]$. Hence $X$ is a closed subspace of $S_{p,p}^\beta (\R^3)$.
By \cref{cor:betaBound} $X$ is also continuously embedded in $L^2(\Omega \times (0,T))$, 
and we can define the following minimization problem:
\begin{equation}
    \fda = \argmin_{f\in X} \Jda(f) := \argmin_{f \in X} \left\{
\frac{1}{2} \ \| \cA f 
    - \gd\|^2_{Y} + \alpha R(f)
\right\},
\label{eq:regularized}
\end{equation}
where $\Jda \colon X \rightarrow \R$ and we choose 
\begin{equation}
    R(f) = \frac{1}{p} \| f \|_X^p = \frac{1}{p} \| \sh  (f) \|^p_{\ell^p(\tilde{w})},
\label{eq:R_pnorm}
\end{equation}
with the weights $\tilde{w}_\lambda = \tilde{w}(j) = 2^{j\beta + \frac{5}{2} j\left( \frac{p}{2}-1\right)}$, being $\lambda \in \Lambda$ the single index describing the shearlet elements, and $j=|\lambda|$ its scale.

We first observe that problem \eqref{eq:regularized} is well-posed, and in particular that the minimizer $\fda$ always exists and is unique. The next result proves this for both proposed noise scenarios.
\begin{proposition}\label{prop:well-posed-regu}
Let $X$ be as in \eqref{eq:set_X} with $1<p<2$ and $\beta \geq \frac{5}{4}(2-p)$, and let $\cA$ be a bounded operator from $X$ to $Y$. Define $\gd = \cA f^\dag + \delta \epsilon$, with $\delta >0$ and $\alpha > 0$. Then:
\begin{itemize}
    \item If $\epsilon \in Y$ satisfies \eqref{eq:noise_det}, there exists a unique solution $\fda$ of \eqref{eq:regularized}.
    \item If $\epsilon$ is a random process on $Y$ satisfying \eqref{eq:noise_stat}, the solution $\fda$ of \eqref{eq:regularized} is a uniquely defined random variable on $X$.
\end{itemize} 
\end{proposition}

\begin{proof}
Let us first consider the deterministic scenario. In this case, the existence of a minimizer is guaranteed by standard arguments, since the functional $\Jda \colon X \rightarrow \R$ is lower semi-continuous and coercive (thanks to \cref{cor:betaBound}), thus its sublevel sets are bounded and, since $X$ is a reflexive Banach space, they are also compact with respect to the weak topology.

Moreover, $\Jda$ is the sum of a convex quadratic term and a strictly convex term (since $p>1$): as a consequence, its minimizer is also unique. \\
In the statistical scenario, at each time the noise $\epsilon(t)$ is a random variable on $\tilde{Y}$, namely, a measurable function on a suitably defined probability space. We can use the same argument as above to show that, for every realization of the process $\epsilon$, a minimizer of $\Jda$ exists in $X$. The function mapping the probability space to the minimizer $\fda$ is also measurable, thanks to Aumann's selection principle (see \cite[Lemma A.3.18]{steinwart2008support}).
\end{proof}

The next step is showing that \eqref{eq:regularized} represents a regularization strategy for the inverse problem \eqref{eq:ProblemSetting}. To do so, we shall prove that if the noise $\delta$ reduces, there exists a selection rule $\alpha(\delta)$ such that the solution $f^\delta_{\alpha(\delta)}$ of the variational problem converges, in a suitable metric, towards the solution $f^\dag$. This is the content of the next section, where we formally state this result, providing also convergence rates.

\subsection{Full measurements: convergence rates}
\label{ssec:nonsampled}

In this section, we provide quantitative estimates regarding the convergence of the regularized solution $\fda$ to the true solution $f^\dag$, according to a suitable a priori choice of parameter $\alpha(\delta)$. Our approach shares many similarities with the one in \cite{burger2018large} (e.g., leading to Corollary 3.12). The main difference with respect to their result is that we are considering a simpler model for the noise, namely that it belongs to $Y$ (deterministic noise) or is a random process in $Y$(statistical noise). One should compare our results with the one in \cite[Corollary 3.12]{burger2018large} by setting $r_2 = 0$ in their estimate.

Since the variational regularization problem \eqref{eq:regularized} involves the minimization of a convex functional, we employ 
tools from convex analysis to address it. In particular, we use the symmetric Bregman distance of the functional $R(f) = \frac{1}{p} \| f \|_X^p$ to measure the distance between $\fda$ and $f^\dag$. We briefly recall its definition. First, for a convex functional $R$, the subdifferential $\partial R(f)$ at a point $f \in X$ consists of:
\[
\partial R (f) = \{ r \in X^*: R(\tilde{f}) \geq R(f) + \langle r, \tilde{f}- f \rangle_* \quad \forall \tilde{f} \in X \},
\]
where $\langle\cdot,\cdot \rangle_*$ denotes the standard pairing between the (reflexive) Banach space $X$ and its dual $X^*$. If the functional $R$ is differentiable (as in our case, since the norm in a reflexive Banach space is Fr\'{e}chet differentiable, see \cite[Remark 2.38]{schuster2012regularization}), then $\partial R(f) = \{ \nabla R(f) \}$. For $f,\tilde{f} \in X$, 
we can define the symmetric Bregman distance between $f$ and $\tilde{f}$ as: 
\[
D_R(\tilde{f},f) = \langle \nabla R(\tilde{f}) - \nabla R(f), \tilde{f} - f \rangle_*.
\]

The purpose of the next result is to show that $D_R(\fda,f^\dag)$ converges to $0$, also providing some convergence rates.
To do so, we need to introduce some additional assumption on the solution $f^\dag$. Within this framework, we limit ourselves to the simplest and most classical assumption, namely the so-called \textit{source condition}:
\begin{equation}
\exists \ w^\dag \in Y \, : \, \nabla R(f^\dag) = \cA^* w^\dag. 
    \label{eq:SC}
\end{equation}
The adjoint operator $\cA^*\colon Y \rightarrow X^*$ is bounded thanks to the assumptions on $\cA$, and is defined by the equality $\langle \cA^* f, g \rangle_{X^*} = \langle f, \cA g\rangle_Y$, for all $f \in Y$ and $g \in X$. 
\begin{proposition}
\label{prop:conv_rate_full}
Let $X$ be as in \eqref{eq:set_X} with $1<p<2$ and $\beta \geq \frac{5}{4}(2-p)$, and let $\cA$ be a bounded operator from $X$ to $Y$. Suppose that $f^\dag$ satisfies the source condition \eqref{eq:SC}. Let $\gd = \cA f^\dag + \delta \epsilon$, $\delta >0$, and $\fda$ the solution of \eqref{eq:regularized} with $\alpha > 0$. Then, if $\epsilon \in Y$ satisfies \eqref{eq:noise_det}, it holds:
\begin{equation}
D_R(\fda,f^\dag) \leq 2\alpha \| w^\dag\|_Y^2 + \frac{2\delta^2}{\alpha}.
    \label{eq:conv_rate_det}
\end{equation}
If $\epsilon$ is a random process on $Y$ satisfying \eqref{eq:noise_stat} it holds:
\begin{equation}
\E[D_R(\fda,f^\dag)] \leq 2\alpha \| w^\dag\|_Y^2 + \frac{2\delta^2}{\alpha}
    \label{eq:conv_rate_stat}
\end{equation}
\end{proposition}
\begin{proof}
Let us first consider the deterministic noise, $\epsilon \in Y$. Since the functional $\Jda$ in \eqref{eq:regularized} is differentiable and strictly convex, the minimizer $\fda$ satisfies the optimality condition:
\[
\cA^*(\cA \fda - \gd) + \alpha \nabla R(\fda) = 0
\]
Taking a duality product with $\fda - f^\dag$ and using the expression of $\gd$, we get
\[
\| \cA (\fda - f^\dag)\|_Y^2 + \alpha \langle \nabla R(\fda), \fda - f^\dag \rangle_* = \delta \langle \cA^* \epsilon, \fda - f^\dag \rangle_*.
\]
Subtracting $\alpha\langle \nabla R(f^\dag), \fda - f^\dag \rangle_*$ on both sides, we get
\[
\| \cA (\fda - f^\dag)\|_Y^2 + \alpha D_R(\fda,f^\dag) = \alpha\langle \nabla R(f^\dag), \fda - f^\dag \rangle_* + \delta \langle \cA^* \epsilon, \fda - f^\dag \rangle_*.
\]
Now, using the source condition and the definition of the adjoint,
\[
\| \cA (\fda - f^\dag)\|_Y^2 + \alpha D_R(\fda,f^\dag) =\langle \alpha w^\dag, \cA (\fda - f^\dag) \rangle_Y + \langle \delta \epsilon, \cA(\fda - f^\dag) \rangle_Y.
\]
Finally, using Young's inequality on both terms on the right-hand side, we get
\begin{equation}
\frac{1}{2}\| \cA (\fda - f^\dag)\|_Y^2 + \alpha D_R(\fda,f^\dag) \leq 2 \alpha^2 \| w^\dag \|_Y^2 + \delta^2 \| \epsilon \|_Y^2,
    \label{eq:final}
\end{equation}
from which \eqref{eq:conv_rate_det} immediately follows. 
In the statistical noise scenario, \eqref{eq:conv_rate_stat} holds for every realization of $\epsilon$: by taking the expectation on both sides, \eqref{eq:conv_rate_stat} follows immediately from \eqref{eq:noise_stat} and from the bound on the moments of sub-Gaussian random variables in terms of their sub-Gaussian norm, see \cite[Proposition 2.5.2]{vershynin2018high}.
\end{proof}
In both estimates, by choosing $\alpha = \alpha(\delta) = \frac{\delta}{\| w^\dag \|}$, we get:
\begin{equation}\label{eq:conv_rate_full}
D_R(\fda,f^\dag)\leq 2 \| w^\dag\|_Y \delta
\qquad \text{and} \qquad \E[D_R(\fda,f^\dag)] \leq 2 \| w^\dag\|_Y \delta.    
\end{equation}
Note that the rates in \cref{prop:conv_rate_full} do not depend on $p$, but the metric $D_R$ does.

\begin{remark}[non-negativity constraint]
In many (static or dynamic) imaging applications, including X-ray tomography, it is known \textit{a priori} that the desired solution $f^\dag$ is non-negative and including this information in the variational problem  \eqref{eq:regularized} is fundamental to obtain superior reconstruction results. This leads to the following minimization problem:
\begin{equation}\label{eq:regularized_NN}
\argmin_{f \in X} \left\{
\frac{1}{2} \ \| \cA f 
    - g^{\delta}\|^2_{Y} + \alpha R(f)+ \iota_+(f)
\right\},
\end{equation}
where $\iota_+$ is the indicator function of the non-negative orthant in $X \subset L^2(\Omega \times(0,T))$, namely:
\[
    \iota_+(f) = \begin{cases}
       0 &\text{if }\;  f \geq 0 \ \text{ a.e. in }\Omega \times (0,T) \\ 
       +\infty &\text{otherwise}.
    \end{cases}
    \]

The statement of \cref{prop:conv_rate_full} still holds in the case of the minimization problem~\eqref{eq:regularized_NN}: it is easy to check that the Bregman distance associated with $R(\cdot)$ is the same of the Bregman distance associated with $R(\cdot)+\iota_+(\cdot)$, when the arguments are all non-negative.
\end{remark}

\subsection{Sampled measurements: convergence rates}
\label{ssec:sampled}

In many practical instances, full measurements are not at disposal, and only a subsampled version of $\gd$ is available. \\
Therefore, we now focus on the case of subsampled measurements, and in particular those sampled according to a random procedure.
In detail, we assume that, at almost every time $t$, instead of the full measurement $(\cA f)(t)$ which should be a function from $\mathcal{U}$ to $\mathcal{V}$, and particularly in $\tilde{Y}=L^2(\mathcal{U};\mathcal{V})$, we only have $N$ point evaluations of it, 
corresponding to the values $u_1(t),\ldots,u_N(t)$ which are randomly sampled.
Since the point evaluation of a function in $\tilde{Y}$ 
is not well-defined, we introduce the following assumptions on the operator $\cA$:
\begin{itemize}
    \item[(O1)] $\cA: X \rightarrow \mathcal{Z}$ where $\mathcal{Z} = L^2(0,T; \mathcal{C}(\mathcal{U},\mathcal{V}))$. 
    \item[(O2)] $\cA$ is continuous from $X$ to $\mathcal{Z}$ and $\| \cA\|_{X \rightarrow \mathcal{Z}} \leq 1$.
\end{itemize}
Assumptions (O1) allows to evaluate $(\cA f)(t)$ at any point $u \in\mathcal{U}$, for $f \in X$, whereas (O2) will be useful later.

Let us consider again our motivating example, introduced in \cref{ssec:dynInvProb} (see Step I, Radon transform). Our goal is to deal with a subsampled version of the sinogram $\mathcal{R}f$, in which the angles are randomly selected.

\begin{MotExample}[Semi-discrete Radon transform] 
The Radon transform $\mathcal{R}$ does not allow for a continuous point-wise sampling in the variable $\theta$. Following \cite{bubba2021,Bubba22}, a possible alternative strategy is to consider a semi-discrete version of the Radon transform (by introducing a discretization of the variable $s$), which ensures higher regularity with respect to the angular variable $\theta$ than the operator $\mathcal{R}$. In particular, we set the variable $s$ in a discrete space which corresponds to modeling the X-ray attenuation measurements with a finite-accuracy detector, consisting of $\Ndtc$ cells. To this end, we introduce a uniform partition $\{I_1, \ldots, I_{\Ndtc}\}$ of the interval $(-\bar{s}, \bar{s})$, where we denote by $s_i$ the midpoint of each interval $I_i$ and take a smooth positive function $\varrho$ of compact support within $(-1,1)$ such that $\int_{-1}^1 \varrho(x) dx = 1$. The semi-discrete Radon transform is then an operator $\radonSD \colon L^2(\Omega) \rightarrow L^2([0,2\pi); \R^{\Ndtc})$ such that, for any $f$ and $\theta$, each component of the vector $(\radonSD f)(\theta) \in \R^{\Ndtc}$ can be written as:
\begin{equation}
	\label{eq:Radon_model_eqL}
	[(\radonSD f)(\theta)]_i = 
	\int_{I_i} \int_{\R} f(s \theta + \sigma \theta^\perp) \varrho\left(\frac{s-s_i}{|I_i|}\right) d\sigma ds= \int_{\R^2} f(x) \varrho_i(x,\theta) dx,
\end{equation}
being $\varrho_i(x,\theta) = \varrho\left(\frac{x_1\cos(\theta)+x_2 \sin(\theta)-s_i}{|I_i|}\right)$. It is easy to check (see \cite[Section 6.1]{bubba2021}) that if $f \in L^2(\Omega)$, then its semi-discrete sinogram $\radonSD f$ belongs to the space $\mathcal{C}((0,2\pi);\R^{\Ndtc})$.

Finally, when considering the dynamic operator $\cradonSD$ associated with the semi-discrete Radon transform $\radonSD$, we introduce local averages in time, to ease the sampling process. To this end, we introduce another smooth, positive, real-valued function $\overline{\varrho}$ such that $\supp(\overline{\varrho})\subset(-\overline{\eta},\overline{\eta})$ and $\int_{-\overline{\eta}}^{\overline{\eta}} \overline{\varrho}(\tau)d\tau = 1$, and define $\cradonSD \colon L^2(\Omega \times (0,T)) \rightarrow Y$ as: 
\begin{equation}
    \label{eq:radonSD_time}
    \quad (\cradonSD f)(t) = \int_0^T \radonSD(f(\cdot,\tau))\overline{\varrho}(t-\tau) d\tau = \radonSD \left(\int_0^T f(\cdot,\tau)\overline{\varrho}(t-\tau) d\tau \right), \quad t \in (0,T).
\end{equation}
It is straightforward to verify that $\cradonSD$ satisfies the assumptions (O1) and (O2) by setting $\mathcal{U}=[0,2\pi)$ and $\mathcal{V} = \R^{N_\text{dtc}}$ for any $X \subset L^2(\Omega \times (0,T))$.

Note that we introduce the smooth partitions of the spatial offset $s$ and time $t$ for mainly technical reasons. However, in practice both of these quantities are always averaged as real detector elements have finite resolution and require some time to collect sufficient photon statistics to make a measurement.

\end{MotExample}

We now specify the statistical measurements model for both sampling and noise. In this section, we only consider a statistical noise scenario. We introduce two random processes (for the sake of clarity, we also show their dependence on the random event $\omega$ living in a suitable probability space). The first random process, $\bu(t,\omega) = (u_n(\omega,t))_{n=1}^N$, is defined in $(0,T)$ and takes values in $\mathcal{U}_N$; the second one, $\be(t,\omega) = (\epsilon_n(\omega,t))_{n=1}^N$, is defined in $(0,T)$ and takes values in $\mathcal{V}_N$. Here, $\mathcal{U}_N$ and $\mathcal{V}_N$ denote the product spaces of $\mathcal{U}$ and $\mathcal{V}$ with themselves $N$ times. We consider the following scalar product in $\mathcal{V}_N$:
\[ \langle {\bf{v}}, {\bf{\tilde{v}}} \rangle_{\mathcal{V}_N} = \frac{1}{N} \sum_{n=1}^N \langle v_n, \tilde{v}_n \rangle_{\mathcal{V}}. \]
Furthermore, we assume the following:
\begin{itemize}
    \item[(M1)] $\be(\omega,\cdot) \in L^2(0,T;\mathcal{V}_N)$ for a.e. $\omega$.
    \item[(M2)] For a.e.~$t \in (0,T)$ and for $n = 1,\ldots,N$, the random variables $\epsilon_n(t)$ are independent and identically distributed as $\nu$, a random variable on $\mathcal{V}$.
    \item[(M3)] $\nu$ is such that $\E[\nu] = 0$ and $\nu$ is sub-Gaussian with $\big\|\| \nu\|_{\mathcal{V}}\big\|_{\text{sG}} \leq 1$.
    \item[(M4)] $\bu(\omega,\cdot)$ is a bounded and measurable function from $(0,T)$ to $\mathcal{U}_N$, for a.e. $\omega$.
    \item[(M5)] For a.e.~$t \in (0,T)$ and for $n = 1,\ldots,N$, the random variables $u_n(t)$ are independent and identically distributed as $\mu$, a random variable on $\mathcal{V}$.
    \item[(M6)] Random variables $\nu$ and $\mu$ are independent.
\end{itemize}

Finally, we can state an analogue of the inverse problem~\eqref{eq:ProblemSetting} for the case of random subsampled measurements:
\begin{equation}\label{eq:ProblemSettingOLD}
\text{given} \; \gd_N \in L^2(0,T;\mathcal{V}_N) \; : \; 
g_N^{\delta}(t) = (\cAbu f^{\dag}) (t) + \delta \be(t) \;\; \text{a.e.}\ t \in (0,T), 
\quad \text{recover} \;  f^{\dag} \in X.
\end{equation}
In particular, the sampling operator:
\[
\cAbu \colon X \rightarrow L^2(0,T;\mathcal{V}_N), \qquad (\cAbu f)(t) = \big((\cA f)(t)(u_n(t))\big)_{n=1}^N \quad \text{a.e. } t \in (0,T)
\]
is well-defined due to our assumptions.
We can regard this operator as the composition $\cAbu = \Sbu \cA$, being $\Sbu\colon \mathcal{Z} \rightarrow L^2(0,T;\mathcal{V}_N)$ the sampling operator associated with the vector $\bu \in \mathcal{U}_N$:
\begin{equation}
    \label{eq:sampling_operator}
    (\Sbu g(t))_n = g(t)(u_n) \quad \text{a.e. } t \in (0,T).
\end{equation}
To solve~\eqref{eq:ProblemSettingOLD} we want to use the following regularization strategy:
\begin{equation}\label{eq:regularizedOLD}
    f_{\alpha,N}^{\delta} = \argmin_{f\in X} \Jdan(f) := \argmin_{f \in X} \left\{
\frac{1}{2} \ \| \cAbu f 
    - g_N^{\delta}\|^2_{L^2(0,T; \mathcal{V}_N)} + \alpha R(f)
\right\},
\end{equation}
with $R$ as in \eqref{eq:R_pnorm}. The existence and uniqueness of $\fdan$ can be proved following the same procedure as in \cref{prop:well-posed-regu}. In particular, thanks to a more refined selection principle suited for product spaces (see \cite[Lemma 6.23]{steinwart2008support}), it can be shown that $\fdan$ is a random variable in $L^2(0,T;\mathcal{V}_N)$, from which the proof follows. 

Similarly to the fully sampled case, we
are interested in deriving convergence estimates for the regularized problem~\eqref{eq:regularizedOLD}, specifically for the quantity $\E[D_R(\fdan,f^\dag)]$. In this context, we can take into account two asymptotic behaviours: when $\delta \rightarrow 0$ and when $N \rightarrow \infty$. In contrast to the (deterministic) inverse problem literature, where one only expects convergence of $\fdan$ to $f^\dag$ when the noise level $\delta$ is reduced, 
in the statistical scenario it is also possible
to recover $f^\dag$ for a non-vanishing noise level, by increasing the number of point evaluations. This is common in the statistical learning literature and is essentially due to the fact that the noise affecting each evaluation is (zero-mean, additive, and) independent: thus, increasing the number of measurements helps cancel the effect of noise. 

In analogy to the fully sampled case, we 
look for an \textit{a priori} choice rule for $\alpha$ of the form $\alpha = \alpha(\delta,N)$. Unfortunately, following closely the simple strategy adopted to prove~\cref{prop:conv_rate_full} leads to convergence rates that are provably suboptimal. 
Hence we propose a strategy which relies more closely on the convexity of the functional $R(f) = \frac{1}{p} \| f \|^p_{X}$, where $X$ is as in \eqref{eq:set_X}, and follows the one in \cite{bubba2021,Bubba22}. 
Indeed, the minimization problem \eqref{eq:regularizedOLD} is analogous to the one considered in \cite{bubba2021,Bubba22}, the main difference being the presence of the Bochner norm $\|\cdot\|_{L^2(0,T;\mathcal{V}_N)}$, which integrates the mismatch in time. As a consequence, we can take advantage of the results proved in \cite{bubba2021,Bubba22}, with slight formal modifications.
For an in-depth comparison
in a non-dynamic setting, we refer to those publications (cf. also \cref{ap:ancillaryRes} where the ancillary results building up to the main result, \cref{thm:general_rate}, are stated).

Recall that $X$ is a reflexive Banach space, continuously embedded in $L^2((0,T)\times \Omega)$ (see \cref{prop:coercivity}) and that 
the convex conjugate $R^\star$ is defined by:
\[
R^\star\colon X^* \rightarrow \R, \quad R^\star(y) = \sup_{x \in X} \left\{\langle y, x \rangle_* - R(x) \right\}.
\]
Next, we define the following quantity, which plays a crucial role to prove the convergence rates and it is strictly connected to source conditions:\begin{equation}\label{eq:Rbeautiful}
	\msR(\bar{\beta}, \bu; f^\dag) = \inf_{\bar w\in L^2(0,T;\mathcal{V}_N)} 
	\bigg\{
	R^\star\left(\sdiff^\dag - \cAbu^* \bar w\right) + \frac{\bar{\beta}}{2} \norm{\bar{w}}_{L^2(0,T;\mathcal{V}_N)}^2
	\bigg\},
\end{equation}
being $\sdiff^\dag = \nabla R(f^\dag)$. This quantity naturally arises in one of the auxiliary results (see the right-hand side of \eqref{eq:aux_breg2} in \cref{prop:aux_convex2}).

We are now ready to state the main result, establishing the convergence rates, for $N \rightarrow \infty$, of the expected value of the Bregman distance associated with the optimal choice of $\alpha$ and for $1<p<2$. 
In particular, the following result has been proved (in the non-dynamic case) in \cite[Theorem 4.10]{bubba2021}, where one has to consider $Q=q/2$. We report it here by first isolating the general estimate \eqref{eq:conv_rate_samp}, which is then specified into the two main convergence rates \eqref{eq:param_choice_p_fixed_Tapio} and \eqref{eq:param_choice_p_redu_Tapio}. To simplify the notation, we use the symbols $\lesssim$ and $\simeq$ to denote that an inequality and an equality hold up to a multiplicative constant independent of $\alpha, \delta, N$.
\begin{theorem}[{\cite[Theorem 4.10 with $Q=q/2$]{bubba2021}}]
\label{thm:general_rate}
Suppose that assumptions (O1)-(O2) are satisfied and it holds that:
\begin{equation}
\label{eq:ass_on_phomog_rate1}
\E \left[ \msR(\bar{\beta}, \bu; f^\dag) \right] \lesssim \bar{\beta} + N^{-\frac{q}{2}}
\qquad
\text{and}
\qquad
\E \left[R^\star(\cAbu^* \be)\right] \lesssim N^{-\frac q2},
\end{equation}
where $N$ is the number of randomly subsampled measurements. Then, as $N \rightarrow \infty$, we have the following estimate:
\begin{equation}
\E \left[D_R(\fdan, f^\dag) \right] \lesssim \alpha + \frac{\delta^2}{N \alpha^3} + N^{-1}
    \label{eq:conv_rate_samp}
\end{equation}
which, depending on the asymptotic regime, reads as:
\begin{itemize}
    \item If $\delta N \rightarrow \infty$ (and $\delta^2/N \rightarrow 0$), then:
    \begin{equation}
	\label{eq:param_choice_p_fixed_Tapio}
	\E \left[D_R(\fdan, f^\dag) \right] \lesssim  \left( \frac{\delta^2}{N}\right)^{\frac{1}{3}} \quad \text{for} \quad \alpha \simeq  \left( \frac{\delta^2}{N}\right)^{\frac{1}{3}}.
    \end{equation}
    \item If $\delta N$ is bounded, then: 
    \begin{equation}
	\label{eq:param_choice_p_redu_Tapio}
	\E \left[D_R(\fdan, f^\dag) \right] \lesssim  N^{-1} \quad \text{for} \quad \alpha \simeq N^{-1}. 
    \end{equation}
\end{itemize}
\end{theorem}

\begin{remark}[non-negativity constraint]
As in the case of full measurements, it is advantageous to incorporate in the variational problem  \eqref{eq:regularizedOLD} the information that the desired solution $f^\dag$ is non-negative. This leads to the following minimization problem:
\begin{equation}\label{eq:regularizedOLD_NN}
\argmin_{f \in X} \left\{
\frac{1}{2} \ \| \cAbu f 
    - g_N^{\delta}\|^2_{L^2(0,T; \mathcal{V}_N)} + \alpha R(f)+ \iota_+(f)
\right\},
\end{equation}
where $\iota_+$ is the indicator function of the non-negative orthant.
The statement of \cref{thm:general_rate} still holds in the case of the minimization problem~\eqref{eq:regularizedOLD_NN}: the proof is analogous to the one outlined in Section 3 of \cite{Bubba22}.
\end{remark}

In order to verify the conditions \eqref{eq:ass_on_phomog_rate1} of \cref{thm:general_rate}, we need to introduce the following assumptions on the exact solution $f^\dag$ and on the operator $\cA$. First, we have to define an operator $\cAm^*$ associated with the probability distribution governing the sample procedure in $\mathcal{U}$. Let $Y_\mu$ be the Hilbert space $Y = L^2(0,T,\tilde{Y}_\mu)$, where $\tilde{Y}_\mu$ is the space $\tilde{Y} = L^2(\mathcal{U};\mathcal{V})$, having replaced its inner product with:
\[
\langle f, g \rangle_\mu = \int_{\mathcal{U}} \langle f(u),g(u)\rangle_{\mathcal{V}}\ d\mu(u).
\]
Then, $\cAm^*$ is the adjoint of $\cA$ in $Y_\mu$. Notice in particular that, if $\mathcal{U}$ is bounded and $\mu$ is a uniform distribution on $\mathcal{U}$, 
 then $\mu$ is a continuous distribution associated with a (constant) density equal to $1/|\mathcal{U}|$. Therefore, the inner products of $Y$ and $Y_\mu$ only differ by a multiplicative constant $1/|\mathcal{U}|$ and $\cAm = 1/|\mathcal{U}| \cA$. 
Next, we introduce the following assumptions:

\begin{itemize}
    \item[(S1)] The ground truth satisfies the 
    source condition: 
    \begin{align*}
    \exists \ w^\dag \in \mathcal{Z} = L^2(0,T;\mathcal{C}(\mathcal{U},\mathcal{V})) \qquad \text{s.t. } \quad  
        \nabla R(f^\dag) = \cAm^* w^\dag.
    \end{align*}
    Notice that this condition is essentially equivalent to the one introduced for the full-measurement case, \eqref{eq:SC}. In particular, there is no need to verify a source condition for every sampled operator $\cAbu$, and the random sampling procedure only affects \eqref{eq:SC} by replacing the original operator $\cA^*$ with its representation $\cAm^*$ in $Y_\mu$.
    \item[(S2)] Let $( \psi_\lambda)_{\lambda \in \Lambda}$ be the cylindrical-shearlet (Parseval) frame introduced in \eqref{eq:cShearSystem} and, for each $\lambda \in \Lambda$, let $X \ni \overline{\psi}_\lambda = \psi_\lambda|_{\Omega \times (0,T)}$. Then the operator $\cA$ satisfies:
    \begin{align*}
        \| (\| \cA \overline{\psi}_\lambda \|_\mathcal{Z} )_\lambda\|_{\ell^q(\hat{w})} = \sum_{\lambda \in \Lambda} \tilde{w}_\lambda^{-q/p} \|\cA \overline{\psi}_\lambda \|_\mathcal{Z}^q
        = C_q < \infty,
    \end{align*}
    where we denoted by $\hat{w}$ the weight sequence such that $\hat{w}_\lambda = \tilde{w}_\lambda^{-q/p}$.
\end{itemize}
These assumptions allow to verify the hypotheses of  \Cref{thm:general_rate}, as shown by the next result.

\begin{proposition}
\label{prop:besov_rate}
The assumptions (S1) and (S2) imply that \eqref{eq:ass_on_phomog_rate1} is satisfied.
\end{proposition}

\begin{proof}
The proof of  \cref{prop:besov_rate} is analogous to the ones of Proposition 4.2 and 4.3 in \cite{Bubba22}, with significant but technical modifications due to the time-dependent framework. To balance self-containment and clarity, we only prove that the first bound in~\eqref{eq:ass_on_phomog_rate1} holds.

    In the definition \eqref{eq:Rbeautiful} of $\msR$, we consider $\bar{w} = \Sbu w^\dag$, being $w^\dag$ the source condition element in (S1) and $\Sbu$ the sampling operator defined in \eqref{eq:sampling_operator}. Thus, we obtain:
\[
\begin{aligned}
    	\msR(\bar{\beta}, \bu; f^\dagger) &\leq
	R^\star(f^\dag - \cAbu^* \Sbu w^\dag) + \frac{\bar{\beta}}{2} \| \Sbu w^\dag \|^2_{L^2(0,T;\mathcal{V}_N)} \\
    & \leq 
    \| \cAm^*w^\dag - \cAbu^* \Sbu w^\dag \|_{X^*}
    + \frac{\bar{\beta}}{2} \| \Sbu w^\dag \|^2_{L^2(0,T;\mathcal{V}_N)},
\end{aligned}
\]
where we used (S1) and \eqref{eq:R_star}. 
The expectation of the second term on the right-hand side can be computed as follows:
\[
\begin{aligned}
    \E_{\bu \sim \mu^N} &\left[\frac{\bar{\beta}}{2} \| \Sbu w^\dag \|^2_{L^2(0,T;\mathcal{V}_N)}\right] = \frac{\bar{\beta}}{2} \int_{\mathcal{U}} \| \Sbu w^\dag(t,\cdot)\|^2_{L^2(0,T;\mathcal{V}_N)} d\mu^N(\bu) \\
    &=  \frac{\bar{\beta}}{2N} \int_{\mathcal{U}} \int_0^T \sum_{n=1}^N \| w^\dag(t,u_n)\|_{\mathcal{V}}^2 dt\ d\mu(u_n)
    =  \frac{\bar{\beta}}{2N} \sum_{n=1}^N  \int_0^T  \| w^\dag(t,\cdot) \|^2_{Y_\mu} = \frac{\bar{\beta}}{2} \| w^\dag\|_{L^2(0,T;Y_\mu)}^2,
\end{aligned}
\]
and gives rise to the first term in the right-hand side of the first bound in \eqref{eq:ass_on_phomog_rate1}. Instead, for the first term on the right-hand side, recalling \eqref{eq:R_star}, we have:
\begin{equation*}
    \begin{aligned}
    \| \cAm^*w^\dag & - \cAbu^* \Sbu w^\dag \|_{X^*} =
    \sum_{\lambda \in \Lambda} \tilde{w}_j^{-q/p} \left| \langle ( \cAm^*w^\dag - \cAbu^* \Sbu w^\dag ) , \overline{\psi}_\lambda \rangle_{X^* \times X} \right|^q \\
    &= \sum_{\lambda \in \Lambda} \tilde{w}_j^{-q/p} \left| \langle w^\dag, \cAm \overline{\psi}_\lambda \rangle_{Y_\mu} - \langle \Sbu w^\dag , \cAbu \overline{\psi}_\lambda \rangle_{L^2(0,T;\mathcal{V}_N)} \right|^q 
    = \sum_{\lambda \in \Lambda} \tilde{w}_j^{-q/p} \left| \frac{1}{N} \sum_{n = 1}^N \zeta_n^\lambda \right|^q 
    \end{aligned}
\end{equation*}
where we have set:
$\zeta^{\lambda}_n =  \langle w^\dag, \cAm \overline{\psi}_{\lambda} \rangle_{Y_\mu} - \langle S_{u_n} w^\dag, S_{u_n} \cA \overline{\psi}_\lambda) \rangle_{L^2(0,T;\mathcal{V})}$.

It is easy to show that the real-valued random variables $\zeta^{\lambda}_n$ are zero-mean and also independent and identically distributed. Moreover, they are bounded random variables. Indeed, in view  of (O2) and of the Cauchy--Schwarz inequality, for any $u_n \in \mathcal{U}$
we have:
\begin{equation*}
\begin{aligned}
    \langle S_{u_n} w^\dag, S_{u_n} \cA \overline{\psi}_\lambda) \rangle_{L^2(0,T;\mathcal{V})} &= \int_0^T \langle w^\dag(t,u_n(t)) ,(\cA \overline{\psi}_\lambda)(t)(u_n(t)) \rangle_{\mathcal{V}} dt \\
    & \leq \int_0^T \| w^\dag(t,\cdot) \|_{\mathcal{C}(\mathcal{U},\mathcal{V})} \| \cA \overline{\psi}_\lambda (t)\|_{\mathcal{C}(\mathcal{U},\mathcal{V})} dt 
    \leq \| w^\dag \|_\mathcal{Z} \| \cA \overline{\psi}_\lambda \|_\mathcal{Z}.
\end{aligned}
\end{equation*}
Therefore, each $\zeta_n^\lambda$ takes values in the interval of size $2 \| w^\dag \|_\mathcal{Z} \| \cA \psi_\lambda \|_\mathcal{Z}$ and we can apply the Hoeffding's inequality for bounded random variables to get:
\[
\Prob\left( \left| \sum_{n=1}^N \zeta_n^\lambda \right| > t \right) \leq 2 \exp \left( -\frac{t^2}{2 \|w^\dag\|_\mathcal{Z}^2 \| \cA \overline{\psi}_\lambda\|_\mathcal{Z}^2}\right).
\]
In conclusion:
\[
\begin{aligned}
\E\left[\| \cAm^*w^\dag - \cAbu^* \Sbu w^\dag \|_{X^*}\right]
&= \frac{1}{q} \sum_{\lambda \in \Lambda} \tilde{w}_j^{-q/p} N^{-q} \E \left[ \left| \sum_{n=1}^N \zeta_n^\lambda \right|^q \right]  \\
& \leq \frac{1}{q} \sum_{\lambda \in \Lambda} \tilde{w}_j^{-q/p} N^{-q} \int_0^\infty t^{q-1} \Prob \left( \left| \sum_{n=1}^N \zeta_n^\lambda \right| > t \right)dt \\
& \leq \frac{2}{q} \sum_{\lambda \in \Lambda} \tilde{w}_j^{-q/p} N^{-q} \int_0^\infty t^{q-1} \exp \left( -\frac{t^2}{2 \|w^\dag\|_\mathcal{Z}^2 \| \cA \overline{\psi}}_\lambda\|_\mathcal{Z}^2\right) dt \\
& \leq \frac{2}{q} \sum_{\lambda \in \Lambda} \tilde{w}_j^{-q/p} N^{-\frac{q}{2}} \|w^\dag\|_\mathcal{Z}^q \|\cA \overline{\psi}_\lambda\|_\mathcal{Z}^q \int_0^\infty s^{q-1}\exp\left( -\frac{1}{2}s^2\right) ds,
\end{aligned}
\]
which can be bounded 
by a constant, depending on $q$ and on $C_q$ in (S2), multiplying $N^{-\frac{q}{2}}$.
\end{proof}

Thanks to \cref{prop:besov_rate}, we know that \cref{thm:general_rate} holds true if (S1) and (S2) are satisfied. The source condition (S1) depends on the unknown solution $f^\dag$ and it is in general difficult to check. Instead, the property (S2) only depends on the operator $\cA$ and on the frame $\{\psi_\lambda\}_\lambda$. For example, this is verified in the case of our motivating example.

\begin{MotExample}[$\cradonSD$ and cylindrical shearlets satisfy (S2)] 
The following lemma shows that, for a cylindrical shearlet frame, condition (S2) is satisfied if, e.g., the forward operator $\cA$ is the semi-discretized Radon transform~\eqref{eq:radonSD_time}.

\begin{lemma}
\label{lem:Apsi}
The cylindrical shearlet frame truncated to $\Omega \times (0,T)$, $\{\overline{\psi}_\lambda \}_\lambda$, and the operator $\cradonSD$ defined by~\eqref{eq:radonSD_time} satisfy
    \begin{equation}
    \sum_{\lambda \in \Lambda} \hat{w}_\lambda \| \cradonSD \overline{\psi}_\lambda \|_\mathcal{Z}^q < \infty.
        \label{eq:Aframe_inf}
    \end{equation}
\end{lemma}
\begin{proof}
First of all, we remark that, since $q\geq 2$ and $\hat{w}_\lambda = \tilde{w}_{\lambda}^{-q/p}$, being $\tilde{w}_\lambda$ uniformly bounded from below, it holds:
\[
\sum_{\lambda \in \Lambda}  \| \cradonSD \overline{\psi}_\lambda \|_\mathcal{Z}^2 < \infty \quad \Rightarrow \quad \sum_{\lambda \in \Lambda} \hat{w}_\lambda \| \cradonSD \overline{\psi}_\lambda \|_\mathcal{Z}^q < \infty.
\]
We thus prove the convergence of the series on the left-hand side. Since $\mathcal{U} = [0,2\pi)$ is one-dimensional, we can take advantage of the Sobolev embedding:

\[
\| \cradonSD \overline{\psi_\lambda}\|^2_\mathcal{Z} = \int_0^T \| (\cradonSD \overline{\psi_\lambda})(t)\|_{C((0,2\pi);\R^{\Ndtc})}^2 dt \lesssim \int_0^T \| (\cradonSD \overline{\psi_\lambda})(t)\|_{H^1((0,2\pi);\R^{\Ndtc})}^2 dt.
\]
We can therefore compute:
\[
\begin{aligned}
   \sum_{\lambda \in \Lambda} \| \cradonSD \overline{\psi_\lambda}\|^2_\mathcal{Z} \lesssim \sum_{\lambda \in \Lambda} \int_0^T \sum_{i=1}^{\Ndtc} \left(\| [(\cradonSD \overline{\psi_\lambda})(t)]_i\|_{L^2(0,2\pi)}^2 + \|\partial_\theta [(\cradonSD \overline{\psi_\lambda})(t)]_i\|_{L^2(0,2\pi)}^2 \right)dt
\end{aligned}
\]
and, more specifically, using the definition of $\cradonSD$ in \eqref{eq:radonSD_time}:
\[
\begin{aligned}
   \sum_{\lambda \in \Lambda} \int_0^T \sum_{i=1}^{\Ndtc} \|[(\cradonSD \overline{\psi_\lambda})(t)]_i\|_{L^2(0,2\pi)}^2 &= \! 
   \sum_{\lambda \in \Lambda} \int_0^T \sum_{i=1}^{\Ndtc}\!\int_0^{2\pi} \left| \int_{0}^T\!\!\!\int_\Omega \psi_\lambda(x,s) \varrho_i(x,\theta) \overline{\varrho}(t-s) dx ds\right|^2\!\!d\theta dt \\
   & = \int_0^T \sum_{i=1}^{\Ndtc} \int_0^{2\pi} \sum_{\lambda \in \Lambda} |\langle \psi_\lambda, \varrho_i(\cdot,\theta)\overline{\varrho}(t-\cdot)\rangle|^2 d\theta dt \\
   & \leq \int_0^T \sum_{i=1}^{\Ndtc} \int_0^{2\pi} \| \varrho_i(\cdot,\theta)\overline{\varrho}(t-\cdot) \|_{L^2(\Omega \times (0,T))}^2 d\theta dt <\infty,
\end{aligned}
\]
and analogously:
\[
\begin{aligned}
   \sum_{\lambda \in \Lambda} \int_0^T \sum_{i=1}^{\Ndtc} \|\partial_\theta[(\cradonSD \overline{\psi_\lambda})(t)]_i\|_{L^2(0,2\pi)}^2 &
   \leq \int_0^T \sum_{i=1}^{\Ndtc} \int_0^{2\pi} \| \partial_\theta\varrho_i(\cdot,\theta)\overline{\varrho}(t-\cdot) \|_{L^2(\Omega \times (0,T))}^2 d\theta dt <\infty.
\end{aligned}
\]

\end{proof}
\end{MotExample}

We now have all the ingredients to numerically verify the expected convergence rates proven in \cref{thm:general_rate} for the two asymptotic behaviours of the noise.

\section{Numerical tests} \label{sec:numericalTests}

To show that the theoretical convergence rates of \cref{ssec:nonsampled} and \cref{ssec:sampled} can be attained in practice, in this section we conduct some numerical experiments (similar to those in~\cite{Bubba22}) which aim at validating the rates. In addition, we provide some of the resulting reconstructions to highlight that while our perspective is relatively theoretical, it also leads to a practical and robust  reconstruction algorithm. 
Since a number of repeated runs of the algorithm is needed to validate the results in expected value (see \cref{ssec:results}), the experiments were carried out on a computational cluster%
\footnote{\url{https://wiki.helsinki.fi/xwiki/bin/view/it4sci/}} %
 using 4 Intel 3.2 GHz CPU cores with 1 GB of RAM each, running MATLAB 2022a on RHEL version 8.

\subsection{Discretization of the mathematical model}
\label{ssec:MathsMod}
The problem setting for the numerical experiments is the one of the motivating example, i.e., a sparse angle dynamic tomography problem with time steps $\{t_1, \ldots, t_\kappa \} \subset [0,T]$, very similar to the one in~\cite{Bubba20, bubba2023efficient}. To set the notation, we briefly explain how to derive a fully discrete counterpart of~\eqref{eq:regularizedOLD} in the case $\cA = \cradonSD$. We discretize objects of $X$ by means of vectors in $\R^{\kappa \Npxl}$, where $\Npxl$ denotes the total number of pixels involved in the discretization of $\Omega$  and $\kappa$ is the number of time steps.
We consider the following discrete model:
\begin{equation}
\gNd = \g_N^\dag + \delta \vec{\epsilon}_N = \RadonD_{\thetab} \f^\dag + \delta \vec{\epsilon}_N
\label{eq:DiscrRadonSampl}
\end{equation}
with 
\begin{equation*}
    \f^\dag = \left[ \begin{array}{c}
    \f_{t_1}^\dag \\
    \vdots \\
    \f_{t_\kappa}^\dag \end{array} \right], \quad \RadonD_{\thetab} = \left[ \begin{array}{ccc}
    \RadonD_{\thetab,t_1} & & \\
     & \ddots & \\
     & & \RadonD_{\thetab,t_\kappa} \end{array} \right], \quad \gNd = \left[ \begin{array}{c}
    \gNd_{,t_1} \\
    \vdots \\
    \gNd_{,t_\kappa} \end{array} \right].
\end{equation*}
where, for $\tau=t_1,\ldots,t_\kappa$, $\f_\tau^\dag \in \R^{\Npxl}$ denotes the (unknown) discrete and vectorized image at time $\tau$,  $\RadonD_{\thetab,\tau} \in \R^{\Ndtc N \times \Npxl}$ represents the sampled version of the Radon operator corresponding to the $N$ randomly sampled angles $\thetab$ at time $\tau$, $\g_{N,\tau}^\dag \in \R^{\Ndtc N}$ is the subsampled sinogram at time $\tau$ and $\vec{\epsilon}_{N,\tau} \in \R^{\Ndtc N}$ is the noise at time $\tau$. In the implementation, for each time step $\tau=t_1,\ldots,t_\kappa$, we consider a normal distribution for the noise vector, $\vec{\epsilon} \sim \mathcal{N}(\vec{0},\bOne_{\Ndtc N})$, where $\bOne_{\Ndtc N}$ is the identity matrix in $\R^{\Ndtc N \times \Ndtc N}$. 
To form each $\RadonD_{\thetab,\tau}$, we pick $N\in \N$ projection angles randomly from a uniform distribution, independently for each time step $\tau$. \\
Then, a regularized solution $\faNd \in \R_+^{\kappa\Npxl}$ is obtained by considering regularizers of the form:
\begin{equation}
\vec{R}(\f) = \frac{1}{p} \norm{\Kop \f}_p^p
\label{eq:BesovReguDiscr}
\end{equation}
where $1 < p < 2$ and $\Kop \in \R^{c\kappa\Npxl \times \kappa\Npxl}$, with $c\geq 1$, depends on the sparsifying transform. Our main focus is cylindrical shearlet-based regularization where $\Kop = \shD$ is the matrix representation of the 3D cylindrical shearlet transform detailed in \cref{ssec:cylShConstruction}. In particular, thanks to the theoretical framework developed in \cref{ssec:DecompSpace}, $\vec{R}(\f)$ is the discrete counterpart to the norm of the smoothness space $S_{p,p}^{\beta}(\R^3)$, where for simplicity we choose the parameter $\beta$ as in \cref{rmrk:noWeights} to have unitary weights in the coefficients.
For comparison, we also consider wavelet-based regularization, where  $\Kop = \Wop$ is the matrix representation associated with the 3D wavelet transform. In this case,  $\vec{R}(\f)$ is equivalent to the norm of $B_p^s(\R^3)$, provided that $s = d\left( \frac{1}{p} -\frac{1}{2} \right)$. 

Even though in this work we do not explicitly derive convergence rates for wavelet-based regularization for the dynamic setting, in view of the results in~\cite{bubba2021, Bubba22}, we expect the theoretical rates to be the same for any orthogonal wavelets. The comparison with wavelet-based regularization is motivated by its wide application in many inverse problems, including tomography, despite being suboptimal for approximating signals of dimension 2 and higher.

Finally, with the notations introduced above, the discrete counterpart of \eqref{eq:regularizedOLD} reads as:
\begin{equation}
\faNd = \argmin_{\f \in \R^{\kappa\Npxl}} \left\{ \frac{1}{2N} \norm{ \RadonD_{\thetab} \f -\gNd }_2^2 + \alpha \vec{R}(\f)  \right\}.   
\label{eq:minsDiscr}
\end{equation}
To solve~\eqref{eq:minsDiscr}, we use the variable metric inexact line-search algorithm (VMILA)~\cite{bonettini2016variable} and the full framework including a list of all the necessary toolboxes is available on Github~\cite{randDynTomo}. In addition, the appendix of \cite{Bubba22} contains a detailed explanation of a numerical implementation process which is almost directly applicable to the dynamic problem in \cref{eq:regularizedOLD} as well. 
The main difference in this case (with respect to the implementation detailed in~\cite{Bubba22}) are the implementation of the forward operator and the corresponding data, which consists of multiple time steps. However, those can be computed separately using standard implementations. In particular, the discrete Radon transform is implemented using the ASTRA Toolbox \cite{van2015astraI, van2016astraII}. Then, we use the cylindrical shearlet transform $\shD$ from \cite{Easley21}, whose implementation is available on Github \cite{3dCylShear}, with a bandlimited generator function, and three scales, with $2^3, 2^3$ and $2^4$ directions per scale, using zero-padding boundary conditions. 
Compared to the original implementation, we chose to precompute the directional filters before starting any of the iterations, which is more efficient.
For wavelet-based regularization, the implementation of the 3D discrete wavelet transform $\Wop$ is from the Matlab wavelet toolbox \cite{waveletToolbox}, using Daubechies-2 filters (with symmetric, i.e., mirrored, boundary conditions) and three scales.

\subsection{Test data} \label{ssec:numericalSetup}
Our regularized reconstruction approach is tested on both simulated and measured data. Specifically, we consider the following two sets of data.
\begin{itemize}
\item {\bf Cartoon phantom}: these are simulated data of a custom ellipsoid phantom (similar to the one used in \cite{bubba2023efficient}). The ground truth at some relevant time steps is shown in \cref{fig:cartoonPhantom}. In particular, the intensity values of the two larger ellipsoids change linearly in the interval $[0, 1]$, while the intensities of the multiple smaller ellipsoids follow a periodic pattern. The boundaries of the phantom remain fixed over time. 
    The spatial dimensions of the phantom are $128 \times 128$ and we simulated in total $\kappa = 32$ sparse angle sinograms (i.e., corresponding to 32 different time steps). For each time step $\tau =t_1, \ldots, t_{32}$, all data projections of $\g^{\delta}_{N,\tau}$ are simulated using parallel-beam geometry from the same reference object $\f_{\tau}^\dagger$ at twice the spatial resolution, and the sinogram is then binned to avoid inverse crime. 
    
\item {\bf STEMPO data set}: these are real X-ray tomography data measured with a motorized device. During the measurement process, a square section of the phantom is translated in a straight line across the fixed circular region. The data are available in Zenodo (cf. \cite{heikkila2022stempoData}) and a detailed description is reported in \cite{heikkila2022stempoDoc}. In particular, here we use the \texttt{stempo\textunderscore seq8x180\textunderscore 2d\textunderscore b8} data, which consists of 1440 projections in fan-beam geometry (8 full rotations with $2^\circ$ equispaced angles), each from a unique time step. We approximate $\kappa = 16$ sinograms by first dividing the data into 16 batches (corresponding to 90 equispaced projections from just a half revolution or a $180^\circ$ range each) and then by linearly interpolating the $N$ random projection angles from each batch. Although the theory presented in \cref{sec:dynTomo} pertains the case of parallel-beam geometry, the extension to the fan-beam case would be a matter of geometrical reparametrization. In practice, the $180^\circ$ range seems to be enough even for fan-beam and allows double the temporal resolution compared to using full $360^\circ$.
The spatial dimensions of the final reconstructions are $280 \times 280$.
    
To verify the convergence rates in expected value of \cref{thm:general_rate}, a ground truth is needed. To obtain a reliable ``ground truth'' in this case, we use the original (densely sampled) sinogram batches to compute a reconstruction using a variational method similar to the one in the original documentation~\cite{heikkila2022stempoDoc}. To avoid bias (towards either cylindrical shearlets or real-valued wavelets priors), we use 3D dual-tree complex wavelets~\cite{chen2011efficient} for the sparsifying transform in the regularization term. The regularization parameter is chosen automatically \cite{purisha2017controlled} and it is not used in any of the other numerical experiments. These reconstructions are shown in \cref{fig:cwt_stempo} and, for selected time steps, on the 
rightmost column of \cref{fig:recons_p32_stempo} and \cref{fig:recons_p1_stempo}.
\end{itemize}

\begin{remark}[Verification of source conditions and cylindrical cartoon-like function assumptions]
The (nearly) optimal approximation result reported in \cref{sec:cylShearlets} holds under the assumption that the target belongs to the class of cylindrical cartoon-like functions (cf.~\cref{def.videos}). In addition, to derive the convergence rates in \cref{sec:dynTomo}, a source condition assumption is needed (cf.~(S1)). However, 
for more realistic simulations, the model of cartoon-like videos~(cf.~\eqref{slant.funct}) offers a more natural setting for the phantoms presented above. Although there is no theoretical guarantee that the same approximation rate holds in this case, we argue that the result of \cref{sec:cylShearlets} still offers a valuable guideline and the numerical results provide an indication that a result similar to~\cref{th:main} is expected to hold in a more general setting.
Regarding condition (S1) for the cartoon phantom, it was already observed in~\cite{Bubba22} that a version of a simulated phantom generated by numerically enforcing the source condition is visually almost indistinguishable from the original. This suggests that the numerical verification of the convergence rates does not differ significantly between the original phantom and its version satisfying the source condition. Instead, for STEMPO data, the source condition cannot be established or enforced because of the experimental nature of the data.
Finally, a relationship has been established between a (different) class of source conditions (in particular, variational source conditions) and the theory of optimally sparse approximations for orthogonal wavelet systems. In~\cite{miller2021maximal}, it was shown that optimally sparse nonlinear wavelet approximations of cartoon-like images automatically satisfy a variational source condition. Although this observation does not apply directly to our setting, it is plausible that a similar relationship could be established also in our case. This topic is beyond the scope of this work and is left for future investigation.
\end{remark}

\begin{figure}[tbh]
    \centering
    \setlength{\imSz}{3.2cm} 
    \setlength{\imShift}{1.2\imSz} 
    
    \begin{tikzpicture} 

        \foreach \t/\j in {08/1,16/2,24/3,32/4}{
            \node[anchor=center] at ($0.03*\imShift*(\j,-1)$) {\includegraphics[width=\imSz]{recons/cartoon_obj_128_t0\t.png}};
            
            \node at ($0.03*\imShift*(\j,-0.45)$) {$\tau = \stripzero{\t}$};
        }
    \end{tikzpicture}
    \caption{Cartoon phantom ground truth at different time steps $\tau$. The reconstruction size is $128{\times}128{\times}32$.}
    \label{fig:cartoonPhantom}
\end{figure}

\begin{figure}[tbh]
    \centering
    
    \setlength{\imSz}{3.2cm} 
    \setlength{\imShift}{1.2\imSz} 
    
    \begin{tikzpicture}
    \foreach \t/\j in {1/1,4/2,7/3,16/4}{
        \node[anchor=center] at ($0.03*\imShift*(\j,-1)$) {\includegraphics[width=\imSz]{recons/stempo_obj_cwt_t\t.png}};

        \node at ($0.03*\imShift*(\j,-0.45)$) {$\tau = \t$};
    }
    \end{tikzpicture}
    \caption{Ground truth reconstruction ($280{\times}280{\times}16$) of the STEMPO phantom at different time steps $\tau$. }
    \label{fig:cwt_stempo}
\end{figure}

\subsection{Results}\label{ssec:results}
In this subsection, we verify the expected convergence rates proven in \cref{thm:general_rate} for the sparse dynamic CT problem \eqref{eq:BesovReguDiscr}-\eqref{eq:minsDiscr}. We focus on the statistical noise regime (cf.~\eqref{eq:noise_stat}), for which similarly to~\cite{bubba2021,Bubba22}  we consider the following two scenarios:
\begin{itemize}
    \item  {\bf Decreasing noise}: in particular, $\delta \simeq N^{-1}$. In this case, the optimal parameter choice is $\alpha \simeq N^{-1}$. In particular, we choose $\delta = c_{\delta}  N_{\min} \|\RadonD \f^\dag\|_{\infty} N^{-1}$, where $c_{\delta} = 0.6$ for the cartoon phantom, and $c_{\delta} = 0.45$ for the STEMPO data. Here, $N_{\min}$ is the minimum value of $N$ considered for the numerical experiments and it is different for the two tested data. Notice that as $N$ increases from $N_{\min}$ to $N_{\max}$, the corresponding noise level $\delta$ decreases from $c_{\delta} \|\RadonD \f^\dag\|_{\infty}$ to $c_{\delta} N_{\min}/N_{\max} \|\RadonD \f^\dag\|_{\infty}$. Finally, we set $\alpha = c_{\alpha}N^{-1}$, where $c_{\alpha}$ is heuristically determined for each choice of data and regularization method. The specific values of $c_\alpha$ and $c_\delta$ are in~\cref{tab:table}. 
    
    \item {\bf Fixed noise}, i.e., $\delta > 0$ constant. Since $\delta N \rightarrow \infty$, we take $\alpha \simeq N^{-1/3}$. In particular, we choose $\delta = c_{\delta} \|\RadonD \f^\dag\|_{\infty}$, where $c_{\delta} = 0.03$ for the cartoon phantom and $c_{\delta} = 0.05$ for the STEMPO data. 
    Then, we let $\alpha = c_{\alpha}N^{-1/3}$, where $c_{\alpha}$ is heuristically determined for each choice of regularization method and data. The specific values of $c_\alpha$ and $c_\delta$ are in \cref{tab:table}.
\end{itemize}

In the cartoon phantom experiments, we test 7 distinct values of $N$ between $N_{\min} = 24$ and $N_{\max} = 240$, and the projection angles are sampled uniformly from the interval $[0,2\pi)$. 
In the STEMPO data experiments, we test 6 distinct values of $N$ between $N_{\min} = 9$ and $N_{\max} = 60$, and the projection angles are sampled uniformly from the interval $[0,\pi)$ or $[\pi, 2\pi)$ in an alternating fashion based on the time step $\tau$. 
Indeed, due to the limited amount of STEMPO data, each of the sinogram $\g^{\delta}_{N,\tau}$ corresponds to just 180$^\circ$ range of projection angles and (unlike parallel-beam) in fan-beam geometry projections from opposite directions  (i.e., $\theta$ and $\theta + 180^\circ$) are unique.

Finally, the expected values appearing in \cref{thm:general_rate} are approximated by sample averages, computed using 5 random realizations. This means that, for each number of angles $N$, the reconstruction is performed 5 times, each time with a different set of $N$ drawn angles (sampled using Matlab’s \texttt{rand}) and a new realization of the noise vector.

\subsubsection{Validation: \texorpdfstring{$p = 3/2$}{p = 3/2}} 

We start by considering the case $p = \frac{3}{2}$, which is fully supported by the theory developed in \cref{sec:dynTomo}. 
We compute the value of the expected Bregman distance $\E \left[D_{\vec{R}}(\faNd, \f^\dag) \right]$ as a function of $N$ and report the slope of the decay rate in \cref{tab:table}. 
In computing the Bregman distance, $\faNd$ and $\f^\dag$ are considered as 3D objects.
In order to provide a quantitative assessment of the decay, we compare the theoretically predicted decay with the experimental one, which is obtained by computing the best monomial approximation $c N^{b}$ of the resulting curves. 

In the supplementary material (\cref{app:supplement}) we have included plots showing the decay rates for both decreasing and fixed noise regimes, using the simulated cartoon phantom in Figure SM1 
and the measured STEMPO data in Figure SM2. 
In each plot, the value of the expected Bregman distance is indicated by a blue solid line and its monomial approximation by a black dashed line. We also report the standard deviation error bars (that is, the shaded region in each plots).
Notice also that, with cylindrical shearlet regularization, there are basically no oscillations around the mean, contrary to the wavelet-based case where, especially with a lower number of projections, the variance is more pronounced.  

All of the predicted and numerically confirmed decay rates are also collected in to \cref{tab:table}, along with the corresponding regularization and noise parameters $c_\alpha$ and $c_\delta$. The first 4 rows (for $p = \frac{3}{2}$) show that the theoretical behaviour is numerically verified: $\E \left[D_{\vec{R}}(\faNd, \f^\dag) \right]$ decays as $N^{-1/3}$ in the fixed noise scenario, and as $N^{-1}$ in the decreasing noise one. 
The results are particularly noteworthy for the STEMPO data considering that the Bregman distance is computed from an estimated ground truth object, obtained from fixed and relatively limited data with some unavoidable measurement error and whose (randomized) projections have been further interpolated.

\begin{table}[bt]
\footnotesize
  \caption{\noindent Theoretical and numerically approximated convergence rates, and different regularization and noise parameter values for each choice of $p$, noise scenario, data set, and regularization method used in the numerical tests.}  \label{tab:table}
\begin{center}
  \begin{tabular}{cllc@{\extracolsep{4pt}}clcll@{}} 
    \multirow{2}{*}{$p$} & \multirow{2}{3.5em}{Noise scenario} & \multirow{2}{*}{Data} & \multirow{2}{5em}{Theoretical rate $b$ \ [$N^b$]} & \multicolumn{2}{c}{Cylindrical shearlet} & \multicolumn{2}{c}{Daubechies-2 wavelet} & \multirow{2}{*}{$c_\delta$} \\ \cline{5-6}\cline{7-8}
       &   &   &   & Approx. rate & $c_\alpha$ & Approx. rate & $c_\alpha$ & \rule[-6pt]{0pt}{0pt}\\ \hline \hline 
    \multirow{4}{*}{\bigger{$\frac{3}{2}$}} & \multirow{2}{*}{\textbf{Decreasing}} & Cartoon & \multirow{2}{*}{\bigger{$-1$}} & -1.190 & 0.03 & -1.001 & 0.12 & 0.6 \rule{0pt}{10pt} \\[0.2em] 
       &   & STEMPO &  & -0.943 & 0.005 & -0.938 & 0.015 & 0.45 \\[.5em] 
       & \multirow{2}{*}{\textbf{Fixed}} & Cartoon & \multirow{2}{*}{\bigger{$-\frac{1}{3}$}} & -0.315 & 0.03 & -0.381 & 0.12 & 0.03 \\[0.2em] 
       &    & STEMPO &   & -0.332 & 0.005 & -0.394 & 0.015 & 0.05 \\[1em] 
    \multirow{4}{*}{\bigger{$1$}} & \multirow{2}{*}{\textbf{Decreasing}} & Cartoon & \multirow{2}{*}{N/A} & -1.053 & 0.001 & -0.996 & 0.004 & 0.6 \\[0.2em] 
       &   & STEMPO &  & -0.920 & 0.0001 & -0.864 & 0.0003 & 0.45 \\[.5em] 
       & \multirow{2}{*}{\textbf{Fixed}} & Cartoon & \multirow{2}{*}{N/A} & -0.320 & 0.001 & -0.373 & 0.004 & 0.03 \\[0.2em] 
       &    & STEMPO &   & -0.340 & 0.00005 & -0.314 & 0.0003 & 0.05 \\ 
  \end{tabular}
\end{center}
\end{table}

Finally, to complement the convergence rates analysis and show that the proposed strategy leads to a practical and robust reconstruction algorithm, we report in \cref{fig:recons_p32_stempo} some reconstructions from the STEMPO data for the fixed noise scenario 
(using both cylindrical shearlets and wavelets). 
All the reconstructions are qualitatively good, especially considering the very limited data. Considering the low data regime, both regularization methods suffer from some sparse-angle artifacts and produce slightly blurry images. However, wavelet-based reconstructions are distinctly noisier. Overall, the cylindrical shearlet-based approach is able to produce clearer reconstructions, even in the limit case of only 9 projection views: the thin circular boundary of the phantom, in particular, is notably sharper. 

Some reconstructions of the simulated data are shown in Figure SM5 
of the supplementary material, including also the relative $\ell^2$ error and SSIM \cite{wang2004image} values. Since the dynamics are less severe, these reconstruction do not show as strong movement artefacts but the quality trends (for the two regularization methods and different number of projections) match those of the real data.

\begin{figure}[th!]
    \centering

    \setlength{\imSz}{2.4cm} 
    \setlength{\imShift}{1.18\imSz} 
    
    \begin{tikzpicture}
    \foreach \mtd/\Method/\offset in {cylSh/Cylindrical shearlets/0, w/Daubechies-2 wavelets/-2.1}{
    \begin{scope}[yshift=\offset*\imSz]
    \node[rotate=90] at ($0.03*\imShift*(0.25,-1.5)$) {\large \Method}; 
    \foreach \t/\j in {1/2,7/1}{
        \foreach \s/\i/\a in {01/1/9, 07/2/13, 13/3/19, 19/4/28}{
            \node[anchor=center] at ($0.03*\imShift*(\i,-\j)$) {\includegraphics[width=\imSz]{recons/stempo_\mtd_fixed_p32_s\s_t00\t.png}};

            \def\tempmtd{cylSh}%
            \ifx \mtd \tempmtd 
                \ifnum \j=1
                    \node at ($0.03*\imShift*(\i,-0.42)$) {\a \ angles};
                \fi
            \fi
        }
        \node[rotate=90] at ($0.03*\imShift*(0.42, -\j)$) {$\tau = \t$};
    }
    \end{scope}
    }
    \begin{scope}[yshift=0*\imSz, xshift=5.118*\imSz]
        \foreach \t/\j in {1/2,7/1}{
            \node[anchor=center] at ($0.03*\imShift*(0,-\j)$) {\includegraphics[width=\imSz]{recons/stempo_obj_cwt_t\t.png}};
            \ifnum \j=1
                \node at ($0.03*\imShift*(0,-0.42)$) {Ground truth};
            \else
                \node[anchor=center] at ($0.03*\imShift*(-0.3,-\j-1.58)$) {\includegraphics[height=2\imSz]{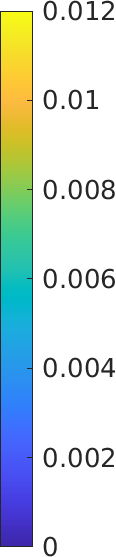}};
            \fi
        }
    \end{scope}
    \end{tikzpicture}
    \caption{Cylindrical shearlet and wavelet reconstructions from selected time steps and angular samples, fixed noise, $p = \frac{3}{2}$. STEMPO data.}
    \label{fig:recons_p32_stempo}
\end{figure}

\subsubsection{Extension: \texorpdfstring{$p = 1$}{p = 1}}
Next, we take a leap of faith and extend the numerical experiments outside the theoretically guaranteed range and set $p = 1$, which corresponds to the widespread regularization approach by minimization of the $\ell^1$-norm of wavelet or shearlet coefficients to enforce sparsity. The numerical machinery from the previous subsection can be still applied, however the Bregman distance is no longer uniquely defined nor is a metric. To circumvent the first issue, as in~\cite{Bubba22}, we follow the strategy in~\cite{burger2007error} to pick an element in the subdifferential $\partial R$. 
The approximate decay rates are listed in \cref{tab:table}.
Despite $D_{\vec{R}}$ still not being a metric, similarly to the numerical tests in~\cite{Bubba22}, it seems that our choice of $D_{\vec{R}}(\faNd, \f^\dag)$
is informative enough to be able to capture the desired convergence properties of the expected Bregman distance $\E \left[D_{\vec{R}}(\faNd, \f^\dag) \right]$. 

The plots showing the resulting decay rates are included in the supplementary material: in Figure SM3 
(simulated cartoon phantom) and  in Figure SM4
(measured STEMPO data).

As in the previous subsection, some reconstructions  from the STEMPO data for the fixed noise scenario are shown in \cref{fig:recons_p1_stempo} 
(using both cylindrical shearlets and wavelets). 
Despite the very limited data, the reconstructions are relatively sharp: both regularization methods are able to reconstruct well the static sections of the phantom and,
with cylindrical shearlets, there are notably less noise and sparse-angle artifacts in the reconstructions.
Instead, the wavelet regularized images show point-like artifacts even with a relatively large number of projections. The cylindrical shearlets produce better results across all $N$, even though the translated square section produces slightly visible trajectory artifacts.

\begin{figure}[th!]
    \centering

    \setlength{\imSz}{2.4cm} 
    \setlength{\imShift}{1.18\imSz} 
    
    \begin{tikzpicture}
    \foreach \mtd/\Method/\offset in {cylSh/Cylindrical shearlets/0, w/Daubechies-2 wavelets/-2.1}{
    \begin{scope}[yshift=\offset*\imSz]
    \node[rotate=90] at ($0.03*\imShift*(0.25,-1.5)$) {\large \Method}; 
    \foreach \t/\j in {1/2,7/1}{
        \foreach \s/\i/\a in {01/1/9, 07/2/13, 13/3/19, 19/4/28}{
            \node[anchor=center] at ($0.03*\imShift*(\i,-\j)$) {\includegraphics[width=\imSz]{recons/stempo_\mtd_fixed_p1_s\s_t00\t.png}};

            \def\tempmtd{cylSh}%
            \ifx \mtd \tempmtd 
                \ifnum \j=1
                    \node at ($0.03*\imShift*(\i,-0.42)$) {\a \ angles};
                \fi
            \fi
        }
        \node[rotate=90] at ($0.03*\imShift*(0.42, -\j)$) {$\tau = \t$};
    }
    \end{scope}
    }
    \begin{scope}[yshift=0*\imSz, xshift=5.118*\imSz]
        \foreach \t/\j in {1/2,7/1}{
            \node[anchor=center] at ($0.03*\imShift*(0,-\j)$) {\includegraphics[width=\imSz]{recons/stempo_obj_cwt_t\t.png}};
            \ifnum \j=1
                \node at ($0.03*\imShift*(0,-0.42)$) {Ground truth};
            \else
                \node[anchor=center] at ($0.03*\imShift*(-0.3,-\j-1.58)$) {\includegraphics[height=2\imSz]{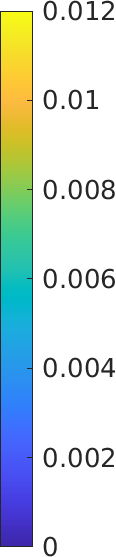}};
            \fi
        }
    \end{scope}
    \end{tikzpicture}
    \caption{Cylindrical shearlet and wavelet reconstructions from selected time steps and angular samples, fixed noise, $p = 1$. STEMPO data.}
    \label{fig:recons_p1_stempo}
\end{figure}

Finally, in \cref{fig:stempo_variance} the pixelwise variance of the different reconstruction samples is shown for both regularization methods and values of $p$. Only one time step ($\tau=7$) is considered, for the sole fixed noise scenario. This illustrates the effect of randomly picking the projection angles for each measurement setting: as one should expect, with only 9 angles there is more variance whereas with 28 projections the reconstruction process is more robust and the results vary less. 
Similar results with 13 and 19 angles are available in the supplementary material (Figure SM6).

For $p=3/2$ (first and second column) the differences are relatively small and overall there is less variance in every test case compared to same instances with $p=1$.
For $p=1$ (third and fourth column) the differences are more noticeable. Cylindrical shearlet-based reconstructions exhibit a larger variance for the dynamic object (rectangle in the middle) while the static parts are more uniform. Instead, wavelet-based reconstructions show significant differences in both the dynamic and static parts of the reconstructions. This was somewhat expected given the larger oscillations of the wavelet samples in Figures SM1 to SM4.


\begin{figure}[th!]
    \centering
    
    \setlength{\imSz}{3cm} 
    \setlength{\imShift}{1.18\imSz} 
    
    \begin{tikzpicture}
    \foreach \p/\offset in {32/0,1/2.1}{
    \begin{scope}[xshift=\offset*\imSz]
    \ifnum \p=32
        \node at ($0.03*\imShift*(1.5,-0.2)$) {\large $p = \frac{3}{2}$};
    \else
        \node at ($0.03*\imShift*(1.5,-0.2)$) {\large $p = 1$};
    \fi
    \foreach \mtd/\Method/\j in {cylSh/Cyl. shearlets/1, w/Wavelets/2}{
        \node[anchor=center] at ($0.03*\imShift*(\j,-0.4)$) {\large \Method}; 
        \foreach \s/\a in {1/9, 2/28}{ 
            \node[anchor=center] at ($0.03*\imShift*(\j,-\s)$) {\includegraphics[width=\imSz]{images/varImgs/stempo_var_\mtd_fixed_p\p_t07_s00\s.png}};

            \def\tempmtd{cylSh}%
            \ifx \mtd \tempmtd 
                \ifnum \p=32
                    \node[rotate=90] at ($0.03*\imShift*(0.43,-\s)$) {\a \ angles};
                \fi
            \fi
        }
    }
    \end{scope}
    }
    \node at ($1.05*\imSz*(0.162,-0.05)$) {\includegraphics[height=2\imSz]{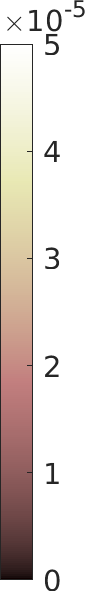}};
    \end{tikzpicture}
    \caption{Pixelwise variance between 5 reconstructions samples of the STEMPO data. Columns show different regularization methods and $p$-norms, rows show different number of projection angles $N$. Only the results from time step $\tau=7$ and fixed noise are shown.}
    \label{fig:stempo_variance}
\end{figure}

\section{Discussion and conclusions} \label{sec:conclusions}
We have demonstrated the application of the cylindrical shearlet transform in variational regularization of spatio-temporal data to solve a class of ill-posed dynamic tomography problems. Our choice is motivated by the 
fact that cylindrical shearlets provide (nearly) optimally sparse approximations for the class of cylindrical cartoon-like functions, which is suitable to model our desired application. 
We also introduce the class of cartoon-like videos, to model more realistic movements, but a proof of (nearly) optimally sparse approximations is left to future work.

Moreover, we extended the theory of shearlet decomposition spaces to include cylindrical shearlets in order to use their decomposition space norms as well-defined regularization functionals with a variety of weights and any $p > 0$.
We examined these regularization strategies and characterized their convergence properties (including  in the statistical inverse learning framework) for dynamic tomography problems. In the full measurements case, we provided convergence rates as the noise level goes to zero, considering both deterministic and random noise conditions.

Then, we considered the situation where only a limited number of imperfect samples are obtained though a random process, in our case a limited number of random X-ray projections of the changing target, and proved bounds for the error decay as the number of measurements grows. We provided rates for two distinct cases: where the amount of noise remains fixed and where its severity decreases with more samples. Our current methods of obtaining the convergence rates are limited to $p > 1$ and an interesting future development would be to show similar rates for $p = 1$ (i.e., the usual choice for sparsity promoting regularization). However, in addition to confirming the theoretical rates, our numerical studies indicate that a similar behavior is likely to be true for $p = 1$.

It is particularly interesting how well most of the numerical experiments follow the predicted theoretical rates with the measured STEMPO data, even if we do not have access to arbitrarily many, truly random and noise free projections nor the ground truth object. In addition, the actual reconstruction algorithm seems robust and effective. Perhaps similar convergence rates are viable also to researchers applying imaging methods (even other than tomography) and not just purely theoretical results.

Although possible, we do no comparisons with traditional shearlets. Firstly, as observed after \cref{th:main}, their approximation properties are expected to fall in between wavelets and cylindrical shearlets but at a higher computational cost; secondly, the relatively small temporal resolution of the data noticeably limits their reliability in practice (especially for the STEMPO data). With cylindrical shearlets this is not an issue as the directional filtering is only performed on the larger, spatial dimensions.

\section*{Acknowledgments}
The authors thank the Finnish Computing Competence Infrastructure (FCCI) for supporting this project with computational and data storage resources.

TAB and LR are members of INdAM-GNCS and
are funded by the European Union - NextGeneration EU
through the Italian Ministry of University and Research as part of the PNRR – M4C2, Investment 1.3 (MUR Directorial Decree no. 341 of 03/15/2022), FAIR ``Future Partnership Artificial Intelligence Research'', Proposal Code PE00000013 - CUP J33C22002830006). 

TH acknowledges the support of the Vilho, Yrjö and Kalle Väisälä Foundation, the Finnish Foundation for Technology Promotion, the Centre of Excellence in Inverse Modelling and Imaging and the Flagship of Advanced Mathematics for Sensing, Imaging and Modelling.

DL acknowledges support of Simons Foundation grant MPS-TSM-00002738.

\appendix

\section{Ancillary results for \texorpdfstring{\cref{sec:dynTomo}}{section 3}}\label{ap:ancillaryRes}

We report here ancillary results leading to \cref{thm:general_rate}. While their statement is analogous to those in~\cite{bubba2021,Bubba22}, we state them here for the specific case of the dynamic setting. We only sketch the proofs since they follow closely the arguments in~\cite{bubba2021}.

The starting point to derive the converge rates in \cref{thm:general_rate} is the deterministic setting, i.e., by fixing a realization of the noise $\be$ and of the sampling pattern $\bu$. We begin with the optimality criterion associated with \eqref{eq:regularized}, which reads as:
\begin{equation}
	\label{eq:optimality_criterion}
	\cAbu^* ( \cAbu \fdan - \gdan) + \alpha \nabla R (\fdan) = 0.
\end{equation}

Since the simple strategy adopted in (the proof of)  \cref{prop:conv_rate_full} 
would deliver suboptimal convergence rates, we use instead the Fenchel-Young's inequalities. Notice that we do not require any source condition on $f^\dag$ yet. This leads to the following result, which is derived exactly as in \cite[proposition 3.2]{bubba2021}, where for simplicity we consider the choice $\Gamma_1 = \widetilde{\gamma}_1 \operatorname{Id}$ and $\Gamma_2 = \widetilde{\gamma}_2 \operatorname{Id}$, where $\widetilde{\gamma}_1, \widetilde{\gamma}_2 \in \R$ are positive constants and $\operatorname{Id}$ is the identity operator. 

\begin{proposition}[{\cite[Proposition 3.2]{bubba2021}}]
\label{prop:aux_convex2}
The regularized solution $\fdan$ given by \eqref{eq:regularized} satisfies
\begin{multline}
	\label{eq:aux_breg2}
	D_R(\fdan, f^\dag) \
	\leq  \inf_{\bar{w} \in L^2(0,T;\mathcal{V}_N)} \left(R^\star\left(\frac{1}{\widetilde{\gamma}_1}(\sdiff^\dag - \cAbu^* \bar w)\right) + \frac \alpha 2 \norm{\bar w}_{L^2(0,T;\mathcal{V}_N)}^2\right)  \\
	 + R(\widetilde{\gamma}_1(f^\dag - \fdan)) + \frac{1}{\alpha} \left(R^\star\left(\frac{\delta} {\widetilde{\gamma}_2} \cAbu^* \be\right)  +  R\left(\widetilde{\gamma}_2(f^\dag - \fdan)\right)\right),
\end{multline}
where $\sdiff^\dag = \nabla R(f^\dag)$ and $\widetilde{\gamma}_1, \widetilde{\gamma}_2 \in \R$ are positive constants.
\end{proposition}

Now, due to the definition of the norm $ \| f \|_{X} = \| \sh f \|_{\ell^p(\tilde{w})}$, we can deduce that, as the space $\ell^p$, $X$ is $p$-smooth and $2$-convex (see \cite[Definitions 2.32-2.33]{schuster2012regularization}).
Hence, to estimate the second and the last term on the right hand side of \eqref{eq:aux_breg2}, as in \cite[Section 4]{bubba2021}, we rely on the $p$-homogeneity of the functional $R$, as well as on tools from convex analysis, such as the Xu-Roach's inequalities~\cite{Xu91} for the $2$-convex Banach space $X$. This leads to the next lemma.
\begin{lemma}[{\cite[Lemma 4.1]{bubba2021}}]
\label{lem:xuroach}
Let $f,\tilde f\in X$. For $1<p<2$ it holds that
\begin{equation}
\label{eq:aux_XuRoach}
	\widetilde{\gamma}^p R(f-\tilde f) \leq C \left(1-\frac p2\right) \widetilde{\gamma}^{\frac{2p}{2-p}} \max\left\{R(f), R(\tilde f)\right\} + \frac p2 D_R(f,\tilde f),
\end{equation}
for some $C>0$ depending on $p$ with any $\widetilde{\gamma}>0$.
\end{lemma}
The first term on the right hand side of \eqref{eq:aux_XuRoach} can be treated through the following 
estimate.
\begin{proposition}[{\cite[Proposition 3.1, second estimate]{bubba2021}}]
\label{prop:apriori}
The unique minimizer $\fdan$ of $\Jdan$ satisfies, for some constant $C>0$,
\begin{equation}
\label{eq:apriori2}
	R(\fdan) \leq C\left(R(f^\dag) + \left(\frac{\delta}\alpha\right)^{\frac{p}{p-1}} R^\star( \cAbu^* \be)\right).
\end{equation}
\end{proposition}

The following result combines all the previous estimates to provide a deterministic bound on the Bregman distance between $\fdan$ and $f^{\dag}$. It relies also on the fact that, for the choice of $R$ outlined in \eqref{eq:R_pnorm}, it holds  that:
\begin{equation} \label{eq:R_star}
    R^\star(y) \leq \frac{1}{q} \| y \|_{X^*}^q = \frac{1}{q} \sum_{\lambda \in \Lambda} \hat{w}_\lambda|\langle y,\overline{\psi}_\lambda\rangle_{X^*\times X}|^q \ ,
\end{equation}
where $\overline{\psi}_\lambda = \psi_\lambda|_{\Omega \times (0,T)}$ are the truncated shearlet frame elements, $q$ is the H\"{o}lder conjugate of $p$ and $\hat{w}_\lambda = (\tilde{w}_\lambda)^{-q/p}$. This can be proved via \cite[Proposition 4.1]{bubba2021}, using the fact that, in this case, the shearlet transform $\sh$ is associated with a Parseval frame, hence $\sh ^\dag = \sh ^*$.

\begin{theorem}[{\cite[Theorem 4.3, (i)]{bubba2021}}]
\label{thm:bregman_dist_gen}
The regularized solution $\fdan$ given by \eqref{eq:regularized} satisfies
\begin{equation*}
	D_R(\fdan, f^\dag) 
	\leq   \widetilde C_p\left[\widetilde{\gamma}_1^{-q} \msR(\alpha \widetilde{\gamma}_1^q, \bu; f^\dag)  + 
	 H(\alpha, \delta, \widetilde{\gamma}_1, \widetilde{\gamma}_2) R^\star(\cAbu^* \be) +
	 \left(\widetilde{\gamma}_1^p	+ \frac{\widetilde{\gamma}_2^p}{\alpha}\right)^{\frac{2}{2-p}} R(f^\dag)\right]
\end{equation*}
for arbitrary $\widetilde{\gamma}_1,\widetilde{\gamma}_2 >0$, where $\widetilde C_p>0$ is a constant independent of $\alpha,\delta$, $N$ and 
\begin{equation*}
	H(\alpha, \delta, \widetilde{\gamma}_1, \widetilde{\gamma}_2) = \frac{\delta^q}{\alpha \widetilde{\gamma}_2^q} + \left(\widetilde{\gamma}_1^p	+ \frac{\widetilde{\gamma}_2^p}{\alpha}\right)^{\frac{2}{2-p}} \left(\frac \delta \alpha\right)^{q}.
\end{equation*}
\end{theorem}

The statement (and the proof) of \cref{thm:general_rate} follows from the last result by taking the expectation of Bregman distance associated with the optimal choice of $\alpha$. 

\section{Supplementary materials} \label{app:supplement}

The supplementary material is available with the published manuscript at the SIAM website: \url{https://doi.org/10.1137/24M1661923}

\bibliographystyle{plain}

{\small%
\bibliography{RandDynTomo}}

\begin{thebibliography}{100}

\bibitem{arridge2020joint}
Simon Arridge, Pascal Fernsel, and Andreas Hauptmann.
\newblock Joint reconstruction and low-rank decomposition for dynamic inverse
  problems.
\newblock {\em Inverse Problems and Imaging}, 16(3):483--523, 2022.

\bibitem{blanchard2018optimal}
Gilles Blanchard and Nicole M{\"u}cke.
\newblock Optimal rates for regularization of statistical inverse learning
  problems.
\newblock {\em Foundations of Computational Mathematics}, 18(4):971--1013,
  2018.

\bibitem{blanke2020inverse}
Stephanie~E. Blanke, Bernadette~N. Hahn, and Anne Wald.
\newblock Inverse problems with inexact forward operator: iterative
  regularization and application in dynamic imaging.
\newblock {\em Inverse Problems}, 36(12):124001, 2020.

\bibitem{bonettini2016variable}
Silvia Bonettini, Ignace Loris, Federica Porta, and Marco Prato.
\newblock Variable metric inexact line-search-based methods for nonsmooth
  optimization.
\newblock {\em SIAM journal on optimization}, 26(2):891--921, 2016.

\bibitem{borup2007frame}
Lasse Borup and Morten Nielsen.
\newblock Frame decomposition of decomposition spaces.
\newblock {\em Journal of Fourier Analysis and Applications}, 13(1):39--70,
  2007.

\bibitem{brandt2024dynamic}
Christina Brandt, Tobias Kluth, Tobias Knopp, and Lena Westen.
\newblock Dynamic image reconstruction with motion priors in application to
  three dimensional magnetic particle imaging.
\newblock {\em SIAM journal on imaging sciences}, 17(3):1539--1586, 2024.

\bibitem{Bubba2018}
T.~A. Bubba, F.~Porta, G.~Zanghirati, and S.~Bonettini.
\newblock {A nonsmooth regularization approach based on shearlets for Poisson
  noise removal in ROI tomography}.
\newblock {\em Appl. Math. Comput.}, 318:131 -- 152, 2018.

\bibitem{bubba2021}
Tatiana~A Bubba, Martin Burger, Tapio Helin, and Luca Ratti.
\newblock Convex regularization in statistical inverse learning problems.
\newblock {\em Inverse Problems and Imaging}, 17(6):1193--1225, 2023.

\bibitem{bubba2023efficient}
Tatiana~A Bubba, Glenn Easley, Tommi Heikkil{\"a}, Demetrio Labate, and Jose~P
  Rodriguez~Ayllon.
\newblock Efficient representation of spatio-temporal data using cylindrical
  shearlets.
\newblock {\em Journal of Computational and Applied Mathematics}, 429:115206,
  2023.

\bibitem{Bubba20}
Tatiana~A Bubba, Tommi Heikkil{\"a}, Hanna Help, Simo Huotari, Yann Salmon, and
  Samuli Siltanen.
\newblock Sparse dynamic tomography: a shearlet-based approach for iodine
  perfusion in plant stems.
\newblock {\em Inverse Problems}, 36(9):094002, 2020.

\bibitem{randDynTomo}
Tatiana~A. Bubba, Tommi Heikkilä, and Luca Ratti.
\newblock {randDynTomo}, 2023.
\newblock Accessed 21.11.2023.

\bibitem{Bubba22}
Tatiana~A Bubba and Luca Ratti.
\newblock Shearlet-based regularization in statistical inverse learning with an
  application to {X}-ray tomography.
\newblock {\em Inverse Problems}, 38(5):054001, 2022.

\bibitem{burger2017variational}
Martin Burger, Hendrik Dirks, Lena Frerking, Andreas Hauptmann, Tapio Helin,
  and Samuli Siltanen.
\newblock A variational reconstruction method for undersampled dynamic x-ray
  tomography based on physical motion models.
\newblock {\em Inverse Problems}, 33(12):124008, 2017.

\bibitem{burger2018large}
Martin Burger, Tapio Helin, and Hanne Kekkonen.
\newblock Large noise in variational regularization.
\newblock {\em Transactions of Mathematics and its Applications}, 2(1):tny002,
  2018.

\bibitem{Burger04}
Martin Burger and Stanley Osher.
\newblock Convergence rates of convex variational regularization.
\newblock {\em Inverse problems}, 20(5):1411, 2004.

\bibitem{burger2007error}
Martin Burger, Elena Resmerita, and Lin He.
\newblock {Error estimation for Bregman iterations and inverse scale space
  methods in image restoration}.
\newblock {\em Computing}, 81:109--135, 2007.

\bibitem{Candes2000}
E.~J. Cand\`{e}s and D.~L. Donoho.
\newblock Curvelets and reconstruction of images from noisy {R}adon data.
\newblock In {\em Proc. SPIE, Wavelet applications in signal and image
  processing VIII}, volume 4119, pages 108--118, 2000.

\bibitem{candes1999}
Emmanuel~Jean Candes, David~Leigh Donoho, et~al.
\newblock {\em Curvelets: A surprisingly effective nonadaptive representation
  for objects with edges}.
\newblock Department of Statistics, Stanford University Stanford, CA, USA,
  1999.

\bibitem{carando2011reconstruction}
Daniel Carando, Silvia Lassalle, and Pablo Schmidberg.
\newblock The reconstruction formula for {B}anach frames and duality.
\newblock {\em Journal of Approximation Theory}, 163(5):640--651, 2011.

\bibitem{casazza1999frames}
Peter~G. Casazza, Deguang Han, and David~R. Larson.
\newblock {Frames for Banach spaces}.
\newblock {\em Contemporary Mathematics}, 247:149--182, 1999.

\bibitem{chambolle1998nonlinear}
Antonin Chambolle, Ronald~A. De~Vore, Nam-Yong Lee, and Bradley~J. Lucier.
\newblock Nonlinear wavelet image processing: variational problems,
  compression, and noise removal through wavelet shrinkage.
\newblock {\em IEEE Trans. Image Process.}, 7(3):319--335, 1998.

\bibitem{chen2019new}
Chong Chen, Barbara Gris, and Ozan \"{O}ktem.
\newblock A new variational model for joint image reconstruction and motion
  estimation in spatiotemporal imaging.
\newblock {\em SIAM Journal on Imaging Sciences}, 12(4):1686--1719, 2019.

\bibitem{chen2011efficient}
Huizhong Chen and Nick Kingsbury.
\newblock Efficient registration of nonrigid 3-d bodies.
\newblock {\em IEEE Trans. Image Process.}, 21(1):262--272, 2011.

\bibitem{Chen2013}
Z.~Chen, X.~Jin, L.~Li, and G.~Wang.
\newblock {A limited-angle CT reconstruction method based on anisotropic TV
  minimization}.
\newblock {\em Phys. Med. Biol.}, 58:2119--41, 2013.

\bibitem{Colonna2010}
F.~Colonna, G.~Easley, K.~Guo, and D.~Labate.
\newblock Radon transform inversion using the shearlet representation.
\newblock {\em Appl. Comput. Harmon. A.}, 29(2):232 -- 250, 2010.

\bibitem{daubechies2004}
I.~Daubechies, M.~Defrise, and C.~De Mol.
\newblock An iterative thresholding algorithm for linear inverse problems with
  a sparsity constraint.
\newblock {\em Comm. Pure Appl. Math.}, 57:1413--1457, 2004.

\bibitem{desbat2007compensation}
Laurent Desbat, S\'ebastien Roux, and Pierre Grangeat.
\newblock Compensation of some time dependent deformations in tomography.
\newblock {\em IEEE Trans. Image Process.}, 26(2):261--269, 2007.

\bibitem{devore1998}
Ronald~A DeVore.
\newblock Nonlinear approximation.
\newblock {\em Acta numerica}, 7:51--150, 1998.

\bibitem{donoho2001}
David~Leigh Donoho.
\newblock Sparse components of images and optimal atomic decompositions.
\newblock {\em Constructive Approximation}, 17(3):353--382, 2001.

\bibitem{3dCylShear}
Glenn Easley, Tommi Heikkilä, et~al.
\newblock {3d\textunderscore cylind\textunderscore shear -- 3d Cylindrical
  Shearlet transform for Matlab}, 2023.
\newblock Accessed 21.11.2023.

\bibitem{Easley21}
Glenn~R Easley, Kanghui Guo, Demetrio Labate, and Basanta~R Pahari.
\newblock Optimally sparse representations of cartoon-like cylindrical data.
\newblock {\em The Journal of Geometric Analysis}, 31(9):8926--8946, 2021.

\bibitem{engl1996regularization}
Heinz~Werner Engl, Martin Hanke, and Andreas Neubauer.
\newblock {\em Regularization of inverse problems}, volume 375.
\newblock Springer Science \& Business Media, 1996.

\bibitem{fall2013dynamic}
Marne~D. Fall, {\'E}ric Barai, Claude Comtat, Thomas Dautremer, et~al.
\newblock Dynamic and clinical {PET} data reconstruction: {A} nonparametric
  {Bayesian} approach.
\newblock In {\em 2013 IEEE ICIP}, pages 345--349, 2013.

\bibitem{feichtinger1987banach}
Hans~G. Feichtinger.
\newblock Banach spaces of distributions defined by decomposition methods,
  {II}.
\newblock {\em Mathematische Nachrichten}, 132(1):207--237, 1987.

\bibitem{feichtinger1985banach}
Hans~G. Feichtinger and Peter Gr{\"o}bner.
\newblock Banach spaces of distributions defined by decomposition methods, {I}.
\newblock {\em Mathematische Nachrichten}, 123(1):97--120, 1985.

\bibitem{Frikel13}
J.~Frikel.
\newblock {Sparse regularization in limited angle tomography}.
\newblock {\em Applied and Computational Harmonic Analysis}, 34(1):117--141,
  January 2013.

\bibitem{gao2011robust}
Hao Gao, Jian-Feng Cai, Zuowei Shen, and Hongkai Zhao.
\newblock Robust principal component analysis-based four-dimensional computed
  tomography.
\newblock {\em Physics in Medicine \& Biology}, 56(11):3181, 2011.

\bibitem{goossens2020}
Bart Goossens, Demetrio Labate, and Bernhard~G Bodmann.
\newblock Robust and stable region-of-interest tomographic reconstruction using
  a robust width prior.
\newblock {\em Inverse Problems and Imaging}, 14(2):291--316, 2020.

\bibitem{goppel2024data}
Simon G{\"o}ppel, J{\"u}rgen Frikel, and Markus Haltmeier.
\newblock {Data-Proximal Complementary $\ell^1$-TV Reconstruction for Limited
  Data Computed Tomography}.
\newblock {\em Mathematics}, 12(10):1606, 2024.

\bibitem{Grasmair11}
Markus Grasmair.
\newblock {Linear convergence rates for Tikhonov regularization with positively
  homogeneous functionals}.
\newblock {\em Inverse Problems}, 27(075014), 2011.

\bibitem{Grasmair08}
Markus Grasmair, Markus Haltmeier, and Otmar Scherzer.
\newblock Sparse regularization with $\ell^q$ penalty term.
\newblock {\em Inverse Problems}, 24(5):055020, 2008.

\bibitem{Grasmair11bis}
Markus Grasmair, Otmar Scherzer, and Markus Haltmeier.
\newblock Necessary and sufficient conditions for linear convergence of
  $\ell^1$-regularization.
\newblock {\em Comm. Pure Appl. Math.}, 64(2):161--182, 2011.

\bibitem{gravier2007tomographic}
Erwan Gravier, Yongyi Yang, and Mingwu Jin.
\newblock Tomographic reconstruction of dynamic cardiac image sequences.
\newblock {\em IEEE transactions on image processing}, 16(4):932--942, 2007.

\bibitem{GL_MMNP}
K.~Guo and D.~Labate.
\newblock The construction of smooth parseval frames of shearlets.
\newblock {\em Mathematical Modelling of Natural Phenomena}, 8(1):82–105,
  2013.

\bibitem{Guo2013}
K.~Guo and D.~Labate.
\newblock {Optimal recovery of 3D X-ray tomographic data via shearlet
  decomposition}.
\newblock {\em Adv. Comput. Math.}, 39(2):227--255, 2013.

\bibitem{guo10}
Kanghui Guo and Demetrio Labate.
\newblock Optimally sparse {3D} approximations using shearlet representations.
\newblock {\em Electronic Research Announcements}, 17:125, 2010.

\bibitem{GL_3D}
Kanghui Guo and Demetrio Labate.
\newblock {Optimally Sparse Representations of 3D Data with $C^2$ Surface
  Singularities Using Parseval Frames of Shearlets}.
\newblock {\em SIAM Journal on Mathematical Analysis}, 44(2):851–886, 2012.

\bibitem{hahn2017motion}
Bernadette~N. Hahn.
\newblock Motion estimation and compensation strategies in dynamic computerized
  tomography.
\newblock {\em Sensing and Imaging}, 18:1--20, 2017.

\bibitem{hakkarainen2019undersampled}
Janne Hakkarainen, Zenith Purisha, Antti Solonen, and Samuli Siltanen.
\newblock Undersampled dynamic {X}-ray tomography with dimension reduction
  {Kalman} filter.
\newblock {\em IEEE Transactions on Computational Imaging}, 5(3):492--501,
  2019.

\bibitem{Haltmeier12}
Markus Haltmeier.
\newblock Stable signal reconstruction via $\ell^{1}$-minimization in
  redundant, non-tight frames.
\newblock {\em IEEE transactions on signal processing}, 61(2):420--426, 2012.

\bibitem{Hamalainen2013}
K.~H\"{a}m\"{a}l\"{a}inen, A.~Kallonen, V.~Kolehmainen, M.~Lassas,
  K.~Niinim\"{a}ki, and S.~Siltanen.
\newblock {Sparse Tomography}.
\newblock {\em SIAM J. Sci. Comput.}, 35(3):B644--B665, 2013.

\bibitem{hampel2022review}
Uwe Hampel, Laurent Babout, Robert Banasiak, Eckhard Schleicher, et~al.
\newblock A review on fast tomographic imaging techniques and their potential
  application in industrial process control.
\newblock {\em Sensors}, 22(6):2309, 2022.

\bibitem{hauptmann2021image}
Andreas Hauptmann, Ozan {\"O}ktem, and Carola Sch{\"o}nlieb.
\newblock Image reconstruction in dynamic inverse problems with temporal
  models.
\newblock {\em Handbook of Mathematical Models and Algorithms in Computer
  Vision and Imaging: Mathematical Imaging and Vision}, pages 1--31, 2021.

\bibitem{heikkila2022stempoDoc}
Tommi Heikkil{\"a}.
\newblock {STEMPO--dynamic X-ray tomography phantom}.
\newblock {\em arXiv preprint arXiv:2209.12471}, 2022.

\bibitem{heikkila2022stempoData}
Tommi Heikkilä.
\newblock {STEMPO - dynamic X-ray tomography phantom}, 2023.
\newblock {v1.2.0 [data set]}.

\bibitem{Hohage2019}
Thorsten Hohage and Philip Miller.
\newblock Optimal convergence rates for sparsity promoting
  wavelet-regularization in {B}esov spaces.
\newblock {\em Inverse Problems}, 35(6):065005, 2019.

\bibitem{isola2008motion}
Alfonso~A. Isola, Andreas Ziegler, Thomas Koehler, Wiro~J. Niessen, and Michael
  Grass.
\newblock Motion-compensated iterative cone-beam {CT} image reconstruction with
  adapted blobs as basis functions.
\newblock {\em Physics in Medicine \& Biology}, 53(23):6777, 2008.

\bibitem{Jia2011}
X.~Jia, B.~Dong, Y.~Lou, and S.~B. Jiang.
\newblock {GPU}-based iterative cone-beam {CT} reconstruction using tight frame
  regularization.
\newblock {\em Phys. Med. Biol.}, 56:3787--807, 2011.

\bibitem{Jorgensen2015}
J.~S. J{\o}rgensen and E.~Y. Sidky.
\newblock {How little data is enough? Phase-diagram analysis of
  sparsity-regularized X-ray computed tomography}.
\newblock {\em Phil. Trans. R. Soc. A}, 373:20140387, 2015.

\bibitem{kadu2023single}
Ajinkya Kadu, Felix Lucka, and Kees~Joost Batenburg.
\newblock Single-shot tomography of discrete dynamic objects.
\newblock {\em arXiv preprint arXiv:2311.05269}, 2023.

\bibitem{kekkonen2014analysis}
Hanne Kekkonen, Matti Lassas, and Samuli Siltanen.
\newblock Analysis of regularized inversion of data corrupted by white
  {G}aussian noise.
\newblock {\em Inverse problems}, 30(4):045009, 2014.

\bibitem{shearlets}
Gitta Kutyniok and Demetrio Labate.
\newblock {\em Shearlets: multiscale analysis for multivariate data}.
\newblock Birkhauser, Applied and Numerical Harmonic Analysis, 2012.

\bibitem{labate13}
Demetrio Labate, Lucia Mantovani, and Pooran Negi.
\newblock Shearlet smoothness spaces.
\newblock {\em Journal of Fourier Analysis and Applications}, 19(3):577--611,
  2013.

\bibitem{lan2023spatiotemporal}
Shiwei Lan, Mirjeta Pasha, and Shuyi Li.
\newblock {Spatiotemporal Besov Priors for Bayesian Inverse Problems}.
\newblock {\em arXiv preprint arXiv:2306.16378}, 2023.

\bibitem{liu2021rethinking}
Jiulong Liu, Angelica~I. Aviles-Rivero, Hui Ji, and Carola-Bibiane
  Sch{\"o}nlieb.
\newblock Rethinking medical image reconstruction via shape prior, going deeper
  and faster: {D}eep joint indirect registration and reconstruction.
\newblock {\em Medical Image Analysis}, 68:101930, 2021.

\bibitem{Liu2012}
Y.~Liu, J.~Ma, Y.~Fan, and Z.~Liang.
\newblock {Adaptive-weighted total variation minimization for sparse data
  toward low-dose X-ray computed tomography image reconstruction}.
\newblock {\em Phys. Med. Biol.}, 57:7923--56, 2012.

\bibitem{Lorenz08}
Dirk~A. Lorenz.
\newblock Convergence rates and source conditions for {T}ikhonov regularization
  with sparsity constraints.
\newblock {\em Journal of Inverse and Ill-posed Problems}, 16, 2008.

\bibitem{Loris2006}
I.~Loris, G.~Nolet, I.~Daubechies, and T.~Dahlen.
\newblock Tomographic inversion using $\ell^1$-norm regularization of wavelet
  coefficients.
\newblock {\em Geophys. J. Int.}, 170(1):359--370, 2007.

\bibitem{lu2002tomographic}
Weiguo Lu and Thomas~R. Mackie.
\newblock Tomographic motion detection and correction directly in sinogram
  space.
\newblock {\em Physics in Medicine \& Biology}, 47(8):1267, 2002.

\bibitem{lunz2021learned}
Sebastian Lunz, Andreas Hauptmann, Tanja Tarvainen, Carola-Bibiane Schonlieb,
  and Simon Arridge.
\newblock On learned operator correction in inverse problems.
\newblock {\em SIAM Journal on Imaging Sciences}, 14(1):92--127, 2021.

\bibitem{milanfar1999model}
Peyman Milanfar.
\newblock A model of the effect of image motion in the {R}adon transform
  domain.
\newblock {\em IEEE Transactions on Image processing}, 8(9):1276--1281, 1999.

\bibitem{miller2021maximal}
Philip Miller and Thorsten Hohage.
\newblock Maximal spaces for approximation rates in $\ell_1$-regularization.
\newblock {\em Numerische Mathematik}, 149(2):341--374, 2021.

\bibitem{waveletToolbox}
Michel Misiti, Georges Oppenheim, Jean-Michel Poggi, Yves Misiti, et~al.
\newblock Wavelet {T}oolbox.
\newblock Release 2022a, original release in 1997.
\newblock MathWorks inc.

\bibitem{mutaf2007impact}
Yildirum~D. Mutaf, John~A. Antolak, and Debra~H. Brinkmann.
\newblock {The impact of temporal inaccuracies on 4DCT image quality}.
\newblock {\em Medical physics}, 34(5):1615--1622, 2007.

\bibitem{myers2014improving}
Glenn~R. Myers, Matthew Geleta, Andrew~M. Kingston, Benoit Recur, and Adrian~P.
  Sheppard.
\newblock Improving dynamic tomography, through maximum a posteriori
  estimation.
\newblock In {\em Developments in X-Ray Tomography IX}, volume 9212, pages
  270--278. SPIE, 2014.

\bibitem{natterer2001mathematical}
Frank Natterer and Frank W{\"u}bbeling.
\newblock {\em Mathematical methods in image reconstruction}.
\newblock SIAM, 2001.

\bibitem{niemi2015dynamic}
Esa Niemi, Matti Lassas, Aki Kallonen, Lauri Harhanen, and other.
\newblock Dynamic multi-source {X}-ray tomography using a spacetime level set
  method.
\newblock {\em J. Comput. Phys.}, 291:218--237, 2015.

\bibitem{Niinimaki2007}
K.~Niinim\"{a}ki, S.~Siltanen, and V.~Kolehmainen.
\newblock {Bayesian multiresolution method for local tomography in dental X-ray
  imaging}.
\newblock {\em Phys. Med. Biol.}, 52:6663--6678, 2007.

\bibitem{Niu2014}
Shanzhou Niu, Yang Gao, Zhaoying Bian, Jing Huang, et~al.
\newblock Sparse-view x-ray {CT} reconstruction via total generalized variation
  regularization.
\newblock {\em Phys. Med. Biol.}, 59:2997--3017, 2014.

\bibitem{oikonomou2016new}
Catherine~M Oikonomou, Yi-Wei Chang, and Grant~J Jensen.
\newblock A new view into prokaryotic cell biology from electron
  cryotomography.
\newblock {\em Nature Reviews Microbiology}, 14(4):205--220, 2016.

\bibitem{otazo2015low}
Ricardo Otazo, Emmanuel Candes, and Daniel~K. Sodickson.
\newblock Low-rank plus sparse matrix decomposition for accelerated dynamic
  {MRI} with separation of background and dynamic components.
\newblock {\em Magnetic resonance in medicine}, 73(3):1125--1136, 2015.

\bibitem{papoutsellis2021core}
Evangelos Papoutsellis, Evelina Ametova, Claire Delplancke, Gemma Fardell,
  et~al.
\newblock {Core Imaging Library-Part II: multichannel reconstruction for
  dynamic and spectral tomography}.
\newblock {\em Philosophical Transactions of the Royal Society A},
  379(2204):20200193, 2021.

\bibitem{pasha2023computational}
Mirjeta Pasha, Arvind~K Saibaba, Silvia Gazzola, Malena~I Espa{\~n}ol, and Eric
  de~Sturler.
\newblock A computational framework for edge-preserving regularization in
  dynamic inverse problems.
\newblock {\em Electronic Transactions on Numerical Analysis}, 58:486--516,
  2023.

\bibitem{purisha2017controlled}
Zenith Purisha, Juho Rimpel{\"a}inen, Tatiana~A. Bubba, and Samuli Siltanen.
\newblock Controlled wavelet domain sparsity for x-ray tomography.
\newblock {\em Measurement Science and Technology}, 29(1):014002, 2017.

\bibitem{rantala2006wavelet}
Maaria Rantala, Simopekka V\"ansk\"a, Seppo J\"arvenp\"a\"a, Martti Kalke,
  Matti Lassas, Jan Moberg, and Samuli Siltanen.
\newblock Wavelet-based reconstruction for limited-angle {X}-ray tomography.
\newblock {\em IEEE transactions on medical imaging}, 25(2):210--217, 2006.

\bibitem{Riis2018}
N.~A.~B. Riis, J.~Fr{\o}sig, Y.~Dong, and P.~C. Hansen.
\newblock {Limited-data x-ray {CT} for underwater pipeline inspection}.
\newblock {\em Inverse Probl.}, 34(3):034002, 2018.

\bibitem{ritchie1994predictive}
Cameron~J. Ritchie, Jiang Hsieh, Michael~F. Gard, J.~David Godwin, Yongmin Kim,
  and Carl~R. Crawford.
\newblock {Predictive respiratory gating: a new method to reduce motion
  artifacts on CT scans.}
\newblock {\em Radiology}, 190(3):847--852, 1994.

\bibitem{ruhlandt2017four}
Aike Ruhlandt, Mareike T{\"o}pperwien, Martin Krenkel, Rajmund Mokso, and Tim
  Salditt.
\newblock Four dimensional material movies: {H}igh speed phase-contrast
  tomography by backprojection along dynamically curved paths.
\newblock {\em Scientific reports}, 7(1):6487, 2017.

\bibitem{schmitt2002efficient}
Uwe Schmitt, Alfred~K. Louis, C.~Wolters, and Marko Vauhkonen.
\newblock Efficient algorithms for the regularization of dynamic inverse
  problems: {II}. {A}pplications.
\newblock {\em Inverse Problems}, 18(3):659, 2002.

\bibitem{schmitzer2020dynamic}
Bernhard Schmitzer, Klaus~P. Schäfers, and Benedikt Wirth.
\newblock {Dynamic Cell Imaging in PET With Optimal Transport Regularization}.
\newblock {\em IEEE Transactions on Medical Imaging}, 39(5):1626--1635, 2020.

\bibitem{lechleiter2018dynamic}
Thomas Schuster, Bernadette Hahn, and Martin Burger.
\newblock Dynamic inverse problems: modelling -- regularization -- numerics.
\newblock {\em Inverse Problems}, 34(040301):1--4, 2018.

\bibitem{schuster2012regularization}
Thomas Schuster, Barbara Kaltenbacher, Bernd Hofmann, and Kamil~S Kazimierski.
\newblock {\em Regularization methods in {Banach} spaces}, volume~10.
\newblock Walter de Gruyter, 2012.

\bibitem{Sydky2006}
E.~Y. Sidky, C.-M. Kao, and X.~Pan.
\newblock {Accurate image reconstruction from few-views and limited-angle data
  in divergent-beam CT}.
\newblock {\em J. X-Ray Sci. Technol.}, 14:119--139, 2006.

\bibitem{Sydky2008}
E.~Y. Sidky and X.~Pan.
\newblock {Image reconstruction in circular cone-beam computed tomography by
  constrained, total-variation minimization}.
\newblock {\em Phys. Med. Biol.}, 53:4777--807, 2008.

\bibitem{steinwart2008support}
Ingo Steinwart and Andreas Christmann.
\newblock {\em Support vector machines}.
\newblock Springer Science \& Business Media, 2008.

\bibitem{tan2015tensor}
Shengqi Tan, Yanbo Zhang, Ge~Wang, Xuanqin Mou, Guohua Cao, Zhifang Wu, and
  Hengyong Yu.
\newblock Tensor-based dictionary learning for dynamic tomographic
  reconstruction.
\newblock {\em Physics in Medicine \& Biology}, 60(7):2803, 2015.

\bibitem{Tian2011}
Z.~Tian, X.~Jia, K.~Yuan, T.~Pan, and S.~B. Jiang.
\newblock Low-dose {CT} reconstruction via edge-preserving total variation
  regularization.
\newblock {\em Phys. Med. Biol.}, 56:5949--67, 2011.

\bibitem{van2016astraII}
Wim Van~Aarle, Willem~J. Palenstijn, Jeroen Cant, Eline Janssens, Folkert
  Bleichrodt, Andrei Dabravolski, Jan De~Beenhouwer, K.~Joost Batenburg, and
  Jan Sijbers.
\newblock Fast and flexible {X}-ray tomography using the {ASTRA} toolbox.
\newblock {\em Optics express}, 24(22):25129--25147, 2016.

\bibitem{van2015astraI}
Wim Van~Aarle, Willem~J. Palenstijn, Jan De~Beenhouwer, et~al.
\newblock The {ASTRA} toolbox: {A} platform for advanced algorithm development
  in electron tomography.
\newblock {\em Ultramicroscopy}, 157:35--47, 2015.

\bibitem{van2012combined}
Geert Van~Eyndhoven, Jan Sijbers, and Joost Batenburg.
\newblock Combined motion estimation and reconstruction in tomography.
\newblock In {\em Computer Vision--ECCV 2012. Workshops and Demonstrations:
  Florence, Italy, October 7-13, 2012, Proceedings, Part I 12}, pages 12--21.
  Springer, 2012.

\bibitem{vandeghinste2013}
Bert Vandeghinste, Bart Goossens, Roel Van~Holen, Christian Vanhove, Aleksandra
  Pi{\v{z}}urica, Stefaan Vandenberghe, and Steven Staelens.
\newblock Iterative {CT} reconstruction using shearlet-based regularization.
\newblock {\em IEEE Transactions on Nuclear Science}, 60(5):3305--3317, 2013.

\bibitem{vershynin2018high}
Roman Vershynin.
\newblock {\em High-dimensional probability: An introduction with applications
  in data science}, volume~47.
\newblock Cambridge university press, 2018.

\bibitem{voigtlaender2023embeddings}
Felix Voigtlaender.
\newblock {\em Embeddings of decomposition spaces}, volume 287.
\newblock American Mathematical Society, 2023.

\bibitem{walker2014vivo}
Simon~M. Walker, Daniel~A. Schwyn, Rajmund Mokso, Martina Wicklein, Tonya
  M{\"u}ller, Michael Doube, Marco Stampanoni, Holger~G. Krapp, and Graham~K.
  Taylor.
\newblock In vivo time-resolved microtomography reveals the mechanics of the
  blowfly flight motor.
\newblock {\em PLoS biology}, 12(3):e1001823, 2014.

\bibitem{wang2014fast}
Kun Wang, Jun Xia, Changhui Li, Lihong~V. Wang, and Mark~A. Anastasio.
\newblock Fast spatiotemporal image reconstruction based on low-rank matrix
  estimation for dynamic photoacoustic computed tomography.
\newblock {\em Journal of biomedical optics}, 19(5):056007--056007, 2014.

\bibitem{wang2004image}
Zhou Wang, Alan~C Bovik, Hamid~R Sheikh, and Eero~P Simoncelli.
\newblock Image quality assessment: from error visibility to structural
  similarity.
\newblock {\em IEEE transactions on image processing}, 13(4):600--612, 2004.

\bibitem{Weidling2020}
Frederic Weidling, Benjamin Sprung, and Thorsten Hohage.
\newblock Optimal convergence rates for {T}ikhonov regularization in {B}esov
  spaces.
\newblock {\em {SIAM} Journal on Numerical Analysis}, 58(1):21--47, 2020.

\bibitem{Xu91}
Zong-Ben Xu and Gary~F Roach.
\newblock Characteristic inequalities of uniformly convex and uniformly smooth
  {B}anach spaces.
\newblock {\em Journal of Mathematical Analysis and Applications},
  157(1):189--210, 1991.

\end{thebibliography}

\end{document}